\documentclass[a4paper,10pt]{amsart}
\pdfoutput=1
\usepackage{amscd}
\usepackage{amssymb}
\usepackage[all]{xy}
\usepackage{graphicx}
\usepackage[T1]{fontenc}
\usepackage{lmodern}
\date{20 October 2014}
\title[BT Operators]{Local Beilinson-Tate Operators}
\usepackage{hyperref}
\hypersetup{colorlinks=false}
\author{Amnon Yekutieli}
\address{Department of  Mathematics, Ben Gurion University,
Be'er Sheva 84105,
Israel}
\email{amyekut@math.bgu.ac.il}
\newtheorem{thm}[equation]{Theorem}
\newtheorem{cor}[equation]{Corollary}
\newtheorem{prop}[equation]{Proposition}
\newtheorem{lem}[equation]{Lemma}
\theoremstyle{definition}
\newtheorem{dfn}[equation]{Definition}
\newtheorem{rem}[equation]{Remark}
\newtheorem{exa}[equation]{Example}

\newtheorem{exer}[equation]{Exercise}
\newtheorem{que}[equation]{Question}

\newtheorem{conj}[equation]{Conjecture}

\numberwithin{equation}{section}
\newcommand{\iso}{\xrightarrow{\simeq}}
\newcommand{\inj}{\hookrightarrow}
\newcommand{\surj}{\twoheadrightarrow}
\newcommand{\xar}{\xrightarrow}
\newcommand{\opn}{\operatorname}
\newcommand{\cat}[1]{\operatorname{\mathsf{#1}}}

\newcommand{\cd}{\,{\cdot}\,}

\newcommand{\ol}{\overline}

\newcommand{\rmitem}[1]{\item[\text{\textup{(#1)}}]}
\newcommand{\mfrak}[1]{\mathfrak{#1}}
\newcommand{\mcal}[1]{\mathcal{#1}}

\newcommand{\mrm}[1]{\mathrm{#1}}
\newcommand{\mbb}[1]{\mathbb{#1}}
\newcommand{\OO}{\mcal{O}}
\newcommand{\MM}{\mcal{M}}

\newcommand{\DD}{\mcal{D}}

\newcommand{\si}{\sigma}

\renewcommand{\th}{\theta}

\newcommand{\al}{\alpha}
\newcommand{\be}{\beta}
\newcommand{\ga}{\gamma}
\newcommand{\ep}{\epsilon}

\newcommand{\Om}{\Omega}

\newcommand{\m}{\mfrak{m}}
\newcommand{\n}{\mfrak{n}}
\renewcommand{\r}{\mfrak{r}}

\newcommand{\kk}{\bsym{k}}

\renewcommand{\k}{\Bbbk}

\newcommand{\Q}{\mathbb{Q}}
\newcommand{\Z}{\mathbb{Z}}

\newcommand{\N}{\mathbb{N}}
\newcommand{\F}{\mathbb{F}}

\newcommand{\tup}[1]{\textup{#1}}
\newcommand{\bsym}[1]{\boldsymbol{#1}}

\newcommand{\ot}{\otimes}

\newcommand{\til}[1]{\tilde{#1}}

\newcommand{\what}[1]{\widehat{#1}}
\newcommand{\bra}[1]{\langle #1 \rangle}
\renewcommand{\d}{\mathrm{d}}
\newcommand{\pa}{\partial}

\newcommand{\lb}{\linebreak}

\newcommand{\bmat}[1]{\begin{bmatrix} #1 \end{bmatrix}}

\begin{document}

\begin{abstract}
In 1968 Tate introduced a new approach to residues on algebraic curves, based 
on a certain ring of operators that acts on the completion at a point of the 
function field of the curve. This approach was generalized to higher 
dimensional algebraic varieties by Beilinson in 1980. However Beilinson's 
paper had very few details, and his operator-theoretic construction remained 
cryptic for many years. Currently there is a renewed interest in the 
Beilinson-Tate approach to residues in higher dimensions. 

Our paper presents a variant of Beilinson's 
operator-theoretic construction. We consider an $n$-dimensional topological 
local field $K$, and define a ring of operators $\opn{E}(K)$ that acts on $K$,
which we call the ring of {\em local Beilinson-Tate operators}.
Our definition is of an analytic nature (as opposed to the original geometric 
definition of Beilinson). We study various properties of the ring 
$\opn{E}(K)$. In particular we show that $\opn{E}(K)$ has an {\em 
$n$-dimensional cubical decomposition}, and this gives rise to a {\em residue 
functional} in the style of Beilinson-Tate. Presumably this residue 
functional coincides with the residue functional that we had constructed 
in 1992; but we leave this as a conjecture. 
\end{abstract}

\maketitle

\tableofcontents

\setcounter{section}{-1}
\section{Introduction}
Let $X$ be a smooth curve over a perfect base field $\k$, with function field  
$\kk(X)$, and let $x \in X$ be a closed point. The completion 
$K := \kk(X)_{(x)}$ of $\kk(X)$ at $x$ is a local field. 
In \cite{Ta}, Tate introduced a ring 
$\opn{E}(K) \subset \opn{End}_{\k}(K)$,
and two-sided ideals 
$\opn{E}(K)_1, \opn{E}(K)_2 \subset \opn{E}(K)$.
These new objects were defined using the valuation ring of $K$.
Let us call the elements of $\opn{E}(K)$ {\em local Tate operators}. 
Heuristically, elements of $\opn{E}(K)_1$ are ``compact operators'', and 
elements of $\opn{E}(K)_2$ are ``discrete operators''.
An operator $\phi \in \opn{End}_{\k}(K)$ is called {\em finite potent} if for 
some positive integer $m$ the operator $\phi^m$ has finite rank. 
Tate proved that each $\phi \in \opn{E}(K)_1 \cap \opn{E}(K)_2$ is 
finite potent, and that 
$\opn{E}(K)_1 + \opn{E}(K)_2 = \opn{E}(K)$.
Using some algebraic manipulations of the structure 
$\bigl( \opn{E}(K), \{ \opn{E}(K)_j \} \bigr)$, 
Tate constructed a {\em residue functional}
\[ \opn{Res}^{\mrm{T}}_{\kk(X) / \k, x} : \Om^1_{\kk(X) / \k} \to  \k . \]
Here $\Om^1_{\kk(X) / \k}$ is the module of K\"ahler $1$-forms of $\kk(X)$. 
Then he showed that his residue functional is the same as the one gotten by 
Laurent series expansion at $x$. 

Finally, Tate gave a global variant of this residue 
functional, using the adeles of $X$ instead of the completion 
$\kk(X)_{(x)}$. He related the local and global residues, and proved that 
when the curve $X$ is proper, the sum of the local residues of any form
$\al \in \Om^1_{\kk(X) / \k}$ is zero.  
The Tate construction gave a totally new way of looking at residues and 
duality for curves. 

This circle of ideas was extended by Beilinson to higher dimensions 
in his extremely brief paper \cite{Be1} (that did not contain any proofs). 
Actually Beilinson's paper had in it several important innovations,
related to a finite type $\k$-scheme $X$. 
By a chain of points in $X$ of length $n$ we mean a sequence 
$\xi = (x_0, \ldots, x_n)$
of points such that $x_{i}$ is a specialization of $x_{i - 1}$.
The chain $\xi$ is {\em saturated} if each  $x_{i}$ is an immediate 
specialization of $x_{i - 1}$. Beilinson said that:
\begin{enumerate}
\item For a chain $\xi$ of length $n$, and a quasi-coherent sheaf $\MM$, 
there is an $\OO_X$-module $\MM_{\xi}$, gotten by an $n$-fold zig-zag inverse 
and direct limit process. When $\MM$ is coherent and $n = 0$ this is the 
$\m_{x_0}$-adic completion $\what{\MM}_{x_0}$ of the stalk $\MM_{x_0}$. 
(Let us call $\MM_{\xi}$ the {\em Beilinson completion} of $\MM$ along $\xi$.)

\item For every $n \in \N$ and quasi-coherent sheaf $\MM$, 
there is a sheaf $\mbb{A}^n(\MM)$ called the {\em sheaf of degree $n$ adeles 
with values in $\MM$}.
It is a restricted product of the Beilinson completions $\MM_{\xi}$ along 
length $n$ chains. The sheaves $\mbb{A}^n(\MM)$ assemble into a flasque 
resolution of $\MM$. When $X$ is a curve, $\mbb{A}^1(\OO_X)$ is the usual sheaf 
of adeles of $X$.

\item For a saturated chain $\xi = (x_0, \ldots, x_n)$, the completion 
$\kk(x_0)_{\xi}$ of the residue field $\kk(x_0)$ is a finite product of 
$n$-dimensional local fields. 

\item Let $A := \kk(x_0)_{\xi}$ as in (3). Then there is a ring 
$\mrm{E}(A) \subset \opn{End}_{\k}(A)$, with an {\em $n$-dimensional cubical 
decomposition} (see Definition \ref{dfn:85} below), from which a 
Tate-style residue functional can be obtained. 

\item The higher adeles from (2) and the cubically
decomposed ring of operators $\mrm{E}(A)$ from (4)
can be combined to prove a global residue theorem when $X$ is proper. 
\end{enumerate}

The adelic resolution (2) was clarified, and all claims proved (for any 
noetherian scheme $X$), by Huber \cite{Hu}. 
The assertion about higher local fields (3) was proved in \cite{Ye1} (for any
excellent noetherian scheme $X$); see Theorem \ref{thm:50}. 

For a long time assertions (4) and (5) were essentially neglected and 
remained cryptic. Very recently we heard about renewed interest in the work of 
Beilinson, mainly by Braunling, Groechenig and Wolfson \cite{Br1, Br2, BGW}. 
The papers \cite{Br1, Br2} discuss item (4) above. A long term goal of this 
team is to understand and make precise the global aspect (5) of Beilinson's 
construction, and then to apply this construction in various directions. 
Independently, Osipov \cite{Os1, Os2} has been studying higher adeles and 
higher local fields. 

Discussions with Wolfson and Braunling led us to the realization that the 
topological aspects of higher local fields, and their implications on item (4) 
above, are not sufficiently understood. 
{\em The purpose of this paper is to present an analytic variant of the 
Beilinson-Tate construction for topological local fields, and to study its 
properties.} Presumably our analytic construction agrees with the 
geometric construction of \cite{Be1, Br2}, and the resulting residue functional 
is the same as residue functional from \cite{Ye1} -- see Conjectures 
\ref{conj:200} and  \ref{conj:204} below. 

Throughout the Introduction we keep the assumption that $\k$ is a perfect base
field. An {\em $n$-dimensional topological local field} 
over $\k$ is -- roughly speaking -- 
a field extension $K$ of $\k$, with a rank $n$ valuation, and with a 
topology compatible with the valuation. An example is the field of iterated 
Laurent series $K = \k((t_2))((t_1))$, which is of dimension $2$. See 
Definitions \ref{dfn:10} and \ref{dfn:12} for details.
It is important to mention that a topological local field $K$ of dimension 
$n \geq 2$ is not a topological ring, but only a {\em semi-topological ring}: 
multiplication is continuous only in one argument. 
We abbreviate ``topological local field'' to ``TLF'' and ``semi-topological'' 
to ``ST''. The theory of ST rings and modules is 
reviewed in Section \ref{sec:ST}.

A TLF $K$ of dimension $n$ has discrete valuation rings $\OO_i(K)$ and residue 
fields $\kk_i(K)$, for $i = 1, \ldots, n$. They are related as follows:
$\kk_i(K)$ is the residue field of $\OO_i(K)$ and the fraction field of 
$\OO_{i + 1}(K)$; and $K$ is the fraction field of $\OO_{1}(K)$.
By a {\em system of liftings} for $K$ we mean a sequence 
$\bsym{\si} = (\si_1, \ldots, \si_n)$, where each 
$\si_i : \kk_i(K) \to \OO_{i}(K)$ is a continuous lifting of the canonical 
surjection $ \OO_{i}(K) \surj \kk_i(K)$. Such systems of liftings always exist; 
see Proposition \ref{prop:150}.

Consider a TLF $K$ equipped with a system of liftings $\bsym{\si}$.
We define a ring of operators 
$\mrm{E}_{\bsym{\si}}(K) \subset \opn{End}_{\k}(K)$, and ideals 
$\mrm{E}_{\bsym{\si}}(K)_{i, j} \subset \mrm{E}_{\bsym{\si}}(K)$.
Our definition (Definitions \ref{dfn:5} and \ref{dfn:73} in the body of the 
paper) is a modification of Beilinson's original definition from \cite{Be1}. 
But whereas Beilinson's original definition was  geometric in nature (and 
pertained only to a TLF arising as a factor of a completion 
$\kk(x_0)_{\xi}$), our definition is of an analytic nature.
(We saw a similar definition in a private communication from Braunling.) 

Here is our first main result. It is repeated as Corollary \ref{cor:100}.

\begin{thm} \label{thm:85}
Let $K$ be an $n$-dimensional TLF over $\k$, and let 
$\bsym{\si}$ and $\bsym{\si}'$ be two systems of liftings for $K$. 
\begin{enumerate}
\item There is equality 
$\mrm{E}_{\bsym{\si}}(K) = \mrm{E}_{\bsym{\si}'}(K)$
of these subrings of $\opn{End}_{\k}(K)$.
\item For any $i = 1, \ldots, n$ and $j = 1, 2$ there is equality 
$\mrm{E}_{\bsym{\si}}(K)_{i, j} = \mrm{E}_{\bsym{\si}'}(K)_{i, j}$
of these ideals of $\mrm{E}_{\bsym{\si}}(K)$.
\end{enumerate}
\end{thm}

The theorem justifies the next definition. 

\begin{dfn} \label{dfn:150}
Let $K$ be an $n$-dimensional TLF over $\k$. Define 
$\mrm{E}(K) := \mrm{E}_{\bsym{\si}}(K)$ and 
$\mrm{E}(K)_{i, j} := \mrm{E}_{\bsym{\si}}(K)_{i, j}$,
where $\bsym{\si}$ is any system of liftings for $K$. 
We call $\mrm{E}(K)$ the ring of {\em local Beilinson-Tate operators} on $K$.
\end{dfn}

Here is a definition from \cite{Br2}, that distills Beilinson's definition from 
\cite{Be1}. The notation we use is closer to the original notation of 
Tate. 

\begin{dfn} \label{dfn:85}
Let $A$ be a commutative $\k$-ring. An {\em $n$-dimensional cubically decomposed 
ring of operators on $A$} is data 
$\bigl( E, \{ E_{i, j} \} \bigr)$ consisting of:
\begin{itemize}
\item A $\k$-subring $E \subset \opn{End}_{\k}(A)$ containing $A$. 

\item Two-sided ideals $E_{i, j} \subset E$, indexed by 
$i \in \{ 1, \ldots, n \}$ and $j \in \{ 1, 2 \}$. 
\end{itemize}
These are the conditions: 
\begin{enumerate}
\rmitem{i} For every $i \in \{ 1, \ldots, n \}$ we have 
$E = E_{i, 1} + E_{i, 2}$.
\rmitem{ii} Every operator 
$\phi \in \bigcap_{i = 1}^n \, \bigcap_{j = 1}^2 \, E_{i, j}$
is finite potent.
\end{enumerate}
\end{dfn}

The next result is Theorem \ref{thm:103}(1) in the body of the paper. 

\begin{thm} \label{thm:86}
Let $K$ be an $n$-dimensional TLF over $\k$. 
The data 
\[ \bigl( \mrm{E}(K), \{ \mrm{E}(K)_{i, j} \} \bigr) \]
of local Beilinson-Tate operators is an $n$-dimensional cubically decomposed 
ring of operators on $K$.
\end{thm}

Let $A$ be any commutative ST $\k$-ring. We can talk about the ring 
$\opn{End}^{\mrm{cont}}_{\k}(A)$ of continuous $\k$-linear operators on $A$. 
There is also the ring $\DD^{\mrm{cont}}_{A / \k}$ of {\em continuous 
differential operators}; see the review in Section \ref{sec:CDOs}.
There are inclusions of $\k$-rings 
\[ A \subset \DD^{\mrm{cont}}_{A / \k} \subset \opn{End}^{\mrm{cont}}_{\k}(A)
\subset \opn{End}_{\k}(A) . \]

\begin{thm} \label{thm:87}
Let $K$ be an $n$-dimensional TLF over $\k$. 
The ring $\mrm{E}(K)$ of local Beilinson-Tate operators satisfies
\[ \DD^{\mrm{cont}}_{K / \k} \subset \mrm{E}(K) \subset
\opn{End}^{\mrm{cont}}_{\k}(K) . \]
\end{thm}

This is repeated as Theorem \ref{thm:103}(2). 
Actually Theorem \ref{thm:87} is used in the proofs of Theorems 
\ref{thm:85} and \ref{thm:86}.

In \cite{Ye1} we developed a theory of residues for TLFs.  For every 
$n$-dimensional TLF $K$ we consider the module of top degree separated 
differential forms $\Om^{n, \mrm{sep}}_{K / \k}$. It is a rank $1$ free 
$K$-module, and it has the fine $K$-module topology. This means that for any 
nonzero form $\al \in \Om^{n, \mrm{sep}}_{K / \k}$, the corresponding 
homomorphism $K \to \Om^{n, \mrm{sep}}_{K / \k}$ is a topological isomorphism. 
(We will say more about the fine topology later in the Introduction.)
The residue  functional
\begin{equation} \label{eqn:85}
\opn{Res}^{\mrm{TLF}}_{K / \k} : \Om^{n, \mrm{sep}}_{K / \k} \to \k 
\end{equation}
constructed in \cite{Ye1} is a continuous $\k$-linear homomorphism, enjoying 
several important properties. See Theorem \ref{thm:70} for details. 

Beilinson \cite{Be1} claimed that an $n$-dimensional cubically decomposed ring 
of operators $\bigl( E, \{ E_{i, j} \} \bigr)$ on 
a commutative $\k$-ring $A$ determines a residue functional
\begin{equation} \label{eqn:88}
\opn{Res}^{\mrm{BT}}_{A / \k; E} : \Om^{n}_{A / \k} \to \k .
\end{equation}
For $n = 1$ this is Tate's original abstract residue from \cite{Ta}. 
For $n \geq 2$ this was worked out in the recent paper \cite{Br2}, using Lie 
algebra homology and Hochschild homology. 

Now consider an $n$-dimensional TLF $K$ over $\k$. 
According to Theorem \ref{thm:86}, $K$ is equipped with an $n$-dimensional 
cubically decomposed ring of operators $\mrm{E}(K)$; and we let
\begin{equation} \label{eqn:103}
\opn{Res}^{\mrm{BT}}_{K / \k} : \Om^{n}_{K / \k} \to \k 
\end{equation}
denote the corresponding residue functional.

Let $A$ be any commutative ST $\k$-ring. For any $q$ the module of K\"ahler 
differentials $\Om^{q}_{A / \k}$ has a canonical topology (this is recalled at 
the end of Section \ref{sec:CDOs}). 
There is a canonical continuous surjection 
$\Om^{q}_{A / \k} \surj \Om^{q, \mrm{sep}}_{A / \k}$
to the separated module of differentials.
Often (e.g.\ when $A = K$ is a TLF of dimension $\geq 1$ and 
$\opn{char}(\k) = 0$) the kernel of 
this canonical surjection is very big. 

\begin{conj} \label{conj:200}
Let $K$ be an $n$-dimensional TLF over $\k$. The following diagram of 
$\k$-linear homomorphisms is commutative:
\[ \UseTips \xymatrix @C=8ex @R=6ex {
\Om^n_{K / \k} 
\ar@{->>}[r]^{\text{ can }}
\ar[dr]_{\opn{Res}^{\mrm{BT}}_{K / \k}}
&
\Om^{n, \mrm{sep}}_{K / \k}
\ar[d]^{\opn{Res}^{\mrm{TLF}}_{K / \k}}
\\
&
\k
} \]
\end{conj}

When $n \leq 1$ we know the conjecture is true. For $n = 0$ it is trivial, and 
for $n = 1$ it is proved in Tate's original paper \cite{Ta}. 
In order to help proving this conjecture in higher dimensions we have included 
a review of the residue functional $\opn{Res}^{\mrm{TLF}}_{K / \k}$ and its 
properties. This is Section \ref{sec:resid} of the paper. We also state 
Conjecture \ref{conj:202}, which is closely related to Conjecture 
\ref{conj:200}.

Suppose $A$ is a finite product of $n$-dimensional TLFs over $\k$, say 
$A = \prod_{l = 1}^r K_l$. Let us define 
$\mrm{E}(A) := \prod_l \mrm{E}(K_l)$ 
and 
$\mrm{E}(A)_{i, j} := \prod_l \mrm{E}(K_l)_{i, j}$.
It is not hard to see that the data
\begin{equation} \label{eqn:100}
\bigl( \mrm{E}(A), \{ \mrm{E}(A)_{i, j} \} \bigr) 
\end{equation}
is an $n$-dimensional cubically decomposed ring of operators on $A$. 

Let $X$ be a finite type $\k$-scheme, let $\xi = (x_0, \ldots, x_n)$ be a 
saturated chain of points in $X$ such that $x_n$ is a closed point, and let 
$K := \kk(x_0)$. According to \cite{Ye1} (see Theorem \ref{thm:50}), the 
Beilinson completion $A := K_\xi$ is a finite product of $n$-dimensional TLFs 
over $\k$. Beilinson's construction from \cite{Be1}, worked out in detail in 
\cite{Br2}, gives rise to 
an $n$-dimensional cubically decomposed ring of operators 
\begin{equation} \label{eqn:101}
\bigl( \mrm{E}_{X, \xi}(K), \{ \mrm{E}_{X, \xi}(K)_{i, j} \} \bigr) \
\end{equation}
on $K_\xi$. (This is our notation.)
Note that by definition both $\mrm{E}_{X, \xi}(K)$ and 
$\mrm{E}(K_{\xi})$ are subrings of $\opn{End}_{\k}(K_{\xi})$.

\begin{conj} \label{conj:204}
Let $X$ be a finite type $\k$-scheme, let $\xi = (x_0, \ldots, x_n)$ be a 
saturated chain of points in $X$ such that $x_n$ is a closed point, and let 
$K := \kk(x_0)$. Then the $n$-dimensional cubically decomposed rings of 
operators $\mrm{E}(K_{\xi})$ and $\mrm{E}_{X, \xi}(K)$ are equal.
\end{conj}

To help proving this conjecture we have included Section \ref{sec:geom}, 
in which we recall some facts from \cite{Ye1} about the Beilinson 
completions $\kk(x_0)_{\xi}$, and provide our interpretation of 
the geometric definition of $\mrm{E}_{X, \xi}(K)$.
In Remark \ref{rem:200} we explain the geometric significance of these 
conjectures. 

To finish the introduction we wish to discuss a technical result that is used 
in the proof of Theorem \ref{thm:85}. This result is of a very general nature, 
and could possibly find other applications. 

We work in the category $\cat{STRing}_{\mrm{c}} \k$ of commutative ST 
$\k$-rings. The morphisms are continuous $\k$-ring homomorphisms.
Let $A \in \cat{STRing}_{\mrm{c}} \k$. The {\em fine $A$-module topology} on 
an $A$-module $M$ is the finest topology that makes $M$ into a ST $A$-module. 
E.g.\ if $M \cong A^r$ for $r \in \N$,  then the product topology is the 
fine $A$-module topology on $M$. For more see Section \ref{sec:ST}.

Consider an artinian local ring $A$ in $\cat{STRing}_{\mrm{c}} \k$, with 
residue field $K$. Give $K$ the fine $A$-module topology relative to 
the canonical surjection $\pi : A \to K$; so $\pi$ becomes a homomorphism in 
$\cat{STRing}_{\mrm{c}} \k$.
Suppose $\si : K \to A$ is a homomorphism in $\cat{STRing}_{\mrm{c}} \k$
such that $\pi \circ \si$ is the identity of $K$. We call $\pi$ a lifting of 
$K$ in $\cat{STRing}_{\mrm{c}} \k$.
The lifting $\si$ is called a {\em precise lifting} if the topology on $A$
coincides with the fine $K$-module topology on it (via $\si$). 
The ring $A$ is called a {\em precise artinian local ring} if 
it admits some precise lifting. There are examples of artinian local rings in 
$\cat{STRing}_{\mrm{c}} \k$ that are not precise, like in Example 
\ref{exa:211}. However, the rings that we are interested in (such as 
quotients of $\OO_1(K)$ for a TLF $K$ -- see Lemma \ref{lem:25}, and Beilinson 
completions of artinian local rings -- see Example \ref{exa:210}) are precise.
The reader might wonder if all continuous liftings $\si : K \to A$ for a 
precise artinian local ring $A$ are precise. This is answered affirmatively in 
Corollary \ref{cor:90} below. It is a consequence of the following technical 
result. 

Given a lifting $\si : K \to A$ and an $A$-module $M$, we denote by 
$\opn{rest}_{\si}(M)$ the $K$-module whose underlying $\k$-module is $M$, and 
$K$ acts via $\si$. 

\begin{thm} \label{thm:90}
Let $A$ be a precise artinian local ring in $\cat{STRing}_{\mrm{c}} \k$, with 
residue field $K$. Put on $K$ the fine $A$-module topology. 
Let $\si_1, \si_2 : K \to A$ be liftings in $\cat{STRing}_{\mrm{c}} \k$
of the canonical surjection $\pi : A \to K$, and assume that $\si_2$ is a 
precise lifting. 

Let $M_1$ and $M_2$ be finite $A$-modules, and let $\phi : M_1 \to M_2$ be an 
$A$-linear homomorphism. For $l = 1, 2$ choose $K$-linear isomorphisms 
$\psi_l : K^{r_l} \iso \opn{rest}_{\si_l}(M_l)$.
Let 
$\bar{\phi} \in \opn{Mat}_{r_2 \times r_1} \bigl( \opn{End}_{\k}(K) \bigr)$
be the matrix representing the $\k$-linear homomorphism
\[ \psi_2^{-1} \circ \phi \circ \psi_1 : K^{r_1} \to K^{r_2} . \]
Then the following hold. 

\begin{enumerate}
\item The matrix $\bar{\phi}$ belongs to 
$\opn{Mat}_{r_2 \times r_1}(\DD^{\mrm{cont}}_{K / \k})$.

\item Assume that $M_1 = M_2$ and $\phi$ is the identity isomorphism. 
Write $r := r_1$. 
Then the matrix $\bar{\phi}$ belongs to 
$\opn{GL}_{r}(\DD^{\mrm{cont}}_{K / \k})$.
\end{enumerate}
\end{thm}

This is repeated as Theorem \ref{thm:20} in the body of the paper. 
{}From it we deduce the next result, which is Corollary \ref{cor:21}.

\begin{cor} \label{cor:90}
Let $A$ be a precise artinian local ring in $\cat{STRing}_{\mrm{c}} \k$,
with residue field $K$. Give $K$ the fine $A$-module topology.
Then any lifting $\si : K \to A$ in $\cat{STRing}_{\mrm{c}} \k$ is a precise 
lifting.
\end{cor}

\medskip \noindent
{\bf Acknowledgments}. I wish to thank Jesse Wolfson and Oliver Braunling
for bringing their work to my attention and for many illuminating 
discussions. Thanks also to Matthew Morrow, John Tate and Alexander Beilinson 
for their input. Lastly, I want to thank the three anonymous referees for 
reading my paper carefully and suggesting many improvements.

%\cleardoublepage
\section{Semi-Topological Rings and Modules}
\label{sec:ST}

We begin with a general discussion of various categories of rings. The notation 
introduced here will make some of our more delicate definitions possible. 

Let $\cat{Ring}$ be the category of rings (not necessarily commutative). 
The morphisms are unit preserving ring homomorphisms. 
Inside it there is the full subcategory $\cat{Ring}_{\mrm{c}}$ of commutative 
rings.  

Now let us fix a nonzero commutative base ring $\k$. A ring homomorphism 
$f : \k \to A$ is called {\em central} if $f(\k)$ is contained in the 
center of $A$. In this case we call $A$ a {\em central $\k$-ring}. (A more 
common name for $A$ is an associative unital $\k$-algebra.)
The central $\k$-rings form a category $\cat{Ring} \k$, in which a morphism
is a ring homomorphism $A \to B$ that respects the given central 
homomorphisms $\k \to A$ and $\k \to B$. 
Inside $\cat{Ring} \k$ there is the full subcategory $\cat{Ring}_{\mrm{c}} \k$ 
of commutative $\k$-rings. Of course if $\k = \Z$ then 
$\cat{Ring} \Z = \cat{Ring}$. 

Let $A$ be a local ring in $\cat{Ring}_{\mrm{c}} \k$, with maximal ideal $\m$ 
and residue field $K = A / \m$. Recall that $A$ is called a complete local ring 
if the canonical homomorphism $A \to \lim_{\leftarrow i} A / \m^i$ is 
bijective. The canonical surjection $\pi : A \surj K$ makes $K$ into an object 
of $\cat{Ring}_{\mrm{c}} \k$.
A {\em lifting} of the canonical surjection $\pi : A \surj K$
in $\cat{Ring}_{\mrm{c}} \k$, or a {\em coefficient field} for $A$ in 
$\cat{Ring}_{\mrm{c}} \k$, is a homomorphism $\si : K \to A$ in 
$\cat{Ring}_{\mrm{c}} \k$, such that the composition $\pi \circ \si$ is the 
identity of $K$. 

The next result is part of the Cohen Structure Theorem. 
We will repeat its proof, because the proof itself will feature in some of our 
constructions. 

\begin{thm}[Cohen] \label{thm:40}
Assume $\k$ is a perfect field. Let $A$ be a complete local ring in 
$\cat{Ring}_{\mrm{c}} \k$, with residue field $K$. 
Then there exists a lifting $\si : K \to A$ in $\cat{Ring}_{\mrm{c}} \k$
of the canonical surjection $\pi : A \surj K$. Moreover, if $\k \to K$ is 
finite, then this lifting $\si$ is unique. 
\end{thm}

\begin{proof}
Consider the $K$-module $\Om^1_{K / \k}$ of K\"ahler differential forms. 
Choose a collection $\{ b_x \}_{x \in X}$ of elements of $K$, 
such that the collection of forms $\{ \d(b_x) \}_{x \in X}$ is a basis of the 
$K$-module $\Om^1_{K / \k}$. 
According to \cite[Theorems 26.5 and 26.8]{Ma},
the collection of elements $\{ b_x \}_{x \in X}$ is algebraically independent 
over $\k$, and $K$ is formally \'etale over the subfield 
$\k(\{ b_x \})$ generated by this collection. 
(Actually, if either $\opn{char} \k = 0$ or $K$ is finitely generated over 
$\k$, then the field $K$ is separable algebraic over $\k(\{ b_x \})$.
Cf.\ \cite[Theorem 26.2]{Ma}.)

For any $x \in X$ choose an arbitrary lifting $\si_{\mrm{b}}(b_x) \in A$ of the 
element $b_x$; namely $\pi(\si_{\mrm{b}}(b_x)) = b_x$. Since the collection 
$\{ b_x \}_{x \in X}$ is algebraically independent over $\k$, the subring 
$\k[\{ b_x \}]$ of $K$ is a polynomial ring. 
Therefore the function $\si_{\mrm{b}} : X \to A$ extends uniquely to a 
homomorphism 
$\si_{\mrm{p}} : \k[\{ b_x \}] \to A$ in $\cat{Ring}_{\mrm{c}} \k$.
Because $A$ is a local ring, for any nonzero element 
$b \in \k[\{ b_x \}]$ its lift $\si_{\mrm{p}}(b)$ is invertible in $A$. 
Thus $\si_{\mrm{p}}$ extends uniquely to a homomorphism 
$\si_{\mrm{r}} : \k(\{ b_x \}) \to A$.
(The subscripts b, p, r refer to ``basis'', ``polynomial'' and ``rational''.)

Let $A_i := A / \m^{i + 1}$, with surjection $\pi_i : A \to A_i$. Because 
$\k(\{ b_x \}) \to K$ is formally \'etale, the homomorphism 
$\pi_i \circ \si_{\mrm{r}} : \k(\{ b_x \}) \to A_i$ 
extends uniquely to a homomorphism 
$\si_i : K \to A_i$, which lifts $A_i \surj K = A_0$. We get an inverse system 
of liftings, and thus a lifting 
$\si : K \to \lim_{\leftarrow i} A / \m^i = A$, 
$\si := \lim_{\leftarrow i} \si_i$.
The restriction of $\si$ to $\k(\{ b_x \})$ equals $\si_{\mrm{r}}$, and in 
particular we see that $\si$ is a homomorphism in $\cat{Ring}_{\mrm{c}} \k$.

If $\k \to K$ is finite then $X = \emptyset$, and hence $\si$ is unique.
\end{proof}

\begin{rem}
Liftings exist whenever they can exist, namely iff $A$ contains a field. 
This is called the equal characteristics case. 
Indeed, if $A$ contains a field then it contains some perfect field $\k$ (e.g.\ 
$\Q$ or $\F_p$). Now Theorem \ref{thm:40} can be applied.
Note that if the residue field $K$ contains $\Q$, then $A$ also contains $\Q$. 

The complication arises when the residue field $K = A / \m$ contains $\F_p$, 
but $A$ does not contain it (i.e.\ $p \neq 0$ in $A$). This is called the mixed 
characteristics case. In this case the notion of lifting has to be modified. 
First the base ring $\k$ is replaced by two rings: a perfect field $\k$ of 
characteristic $p$, and a complete DVR $\til{\k}$, whose maximal ideal is 
generated by $p$, and its residue field is $\k$. 
The ring $\til{\k}$ is called the ring of Witt vectors of $\k$. 
(E.g.\ when $\k = \mbb{F}_p$, its ring of Witt vectors is 
$\til{\k} = \what{\Z}_{(p)}$, the $p$-adic integers.) 
A homomorphism $\k \to K$ lifts canonically to a homomorphism 
$\til{\k} \to A$. Next there is a complete DVR $\til{K}$, whose maximal ideal is 
generated by $p$, its residue field is $K$, and $\til{\k} \to \til{K}$ is 
$p$-adically formally smooth. Therefore there exists a 
lifting $\si : \til{K} \to A$ over $\til{\k}$. Moreover, all such 
liftings are controlled by $\Om^1_{K / \k}$, just as in the proof of Theorem 
\ref{thm:40}.

In this paper we shall be exclusively concerned with the equal characteristics 
case. 
\end{rem}

We are going to look at a more subtle lifting situation, involving topologies 
on $A$ and $K$. 

We consider the base ring $\k$ as a topological ring with the discrete 
topology. Recall that a topological $\k$-module is a $\k$-module $M$, 
endowed with a topology, such that addition and multiplication are 
continuous functions $M \times M \to M$ and $\k \times M \to M$ 
respectively. We say that the topology on $M$ is {\em $\k$-linear}, and that 
$M$ is a {\em linearly topologized $\k$-module}, if the element $0 \in M$ 
has a basis of open neighborhoods consisting of open $\k$-submodules. 

In order to define a $\k$-linear topology on a $\k$-module $M$, all we have to 
do is to provide a collection $\{ U_i \}_{i \in I}$ of $\k$-submodules 
of $M$, that is cofiltered under inclusion; namely for any $i, j \in I$ there 
exists $k \in I$ such that $U_k \subset U_i \cap U_j$.
The resulting topology on $M$, in which the collection $\{ U_i \}_{i \in I}$
is a basis of open neighborhoods of $0 \in M$, is called the $\k$-linear 
topology generated by this collection.

\begin{dfn}
Let $M_1, \ldots, M_p, N$ be  linearly topologized $\k$-modules, and let 
$\mu : \prod_{i = 1}^p M_i \to N$ be a $\k$-multilinear function. We say that 
$\mu$ is {\em semi-continuous} if for every 
$\bsym{m} = (m_1, \ldots, m_p) \in \prod_{i = 1}^p M_i$ and every 
$i \in \{ 1, \ldots, p \}$, the homomorphism 
\[ \mu_{\bsym{m}, i} : M_i \to N , \quad 
\mu_{\bsym{m}, i}(m'_i) := 
\mu(m_1, \ldots, m_{i - 1}, m'_i, m_{i + 1},  \ldots, m_p), \]
is continuous. 
\end{dfn}

Here is a definition from \cite{Ye1}. 

\begin{dfn}
A {\em semi-topological $\k$-ring} is a $\k$-ring $A$, with a $\k$-linear 
topology on it (so the underlying $\k$-module of $A$ is a linearly 
topologized $\k$-module), such that multiplication 
$\mu : A \times A \to A$  is a semi-continuous bilinear function.

The semi-topological $\k$-rings form a category $\cat{STRing} \k$, in which the 
morphisms are the continuous $\k$-ring homomorphisms.
\end{dfn}

We use the abbreviation ``ST'' for ``semi-topological''. 
The ring $\k$ with its discrete topology is the 
initial object of $\cat{STRing} \k$.
Inside $\cat{STRing} \k$ there is the full subcategory 
$\cat{STRing}_{\mrm{c}} \k$ of commutative ST $\k$-rings.

\begin{exa}
Suppose $A$ is a commutative $\k$-ring, and $\mfrak{a} \subset A$ is an ideal. 
Give $A$ the $\mfrak{a}$-adic topology. Then $A$ is a ST $\k$-ring. 
(The ring $A$ is actually a topological ring, namely multiplication 
$A \times A \to A$ is continuous.) 
The ring of Laurent series $A((t))$ -- see Definition \ref{dfn:41} -- is a ST 
$\k$-ring, but it is usually not a topological ring.   
\end{exa}

\begin{dfn}
Let $A$ be a ST $\k$-ring. A {\em left ST $A$-module} is a left 
$A$-module $M$, endowed with a $\k$-linear topology on it (so $M$ is a linearly 
topologized $\k$-module), such that the bilinear function 
$\mu : A \times M \to M$, $\mu(a, m) := a \cd m$, 
is semi-continuous.

The ST left $A$-modules form a category, in which the morphisms are the 
continuous $A$-linear homomorphisms. We denote this category by 
$\cat{STMod} A$. 
\end{dfn}

There is a similar right module version, denoted by $\cat{STMod} A^{\mrm{op}}$.

\begin{rem}
If $A$ is a discrete ST $\k$-ring (e.g.\ $A = \k$), then a ST $A$-module $M$ is 
also a topological $A$-module; namely the multiplication function 
$A \times M \to M$ is continuous. We will usually ignore this fact. 
\end{rem}

\begin{prop} \label{prop:41}
Let $A$ be a ST $\k$-ring. The category $\cat{STMod} A$ has these 
properties\tup{:}
\begin{enumerate}
\item It is a $\k$-linear additive category.
\item It has limits and colimits \tup{(}of arbitrary cardinality\tup{)}. 
In particular there are coproducts, products, kernels and cokernels.
\end{enumerate}
\end{prop}

\begin{proof}
This is all essentially in \cite[Section 1.2]{Ye1}. The fact that 
$\cat{STMod} A$ is $\k$-linear is clear. 

Given a collection $\{ M_x \}_{x \in X}$ of ST $A$-modules, indexed by a set 
$X$, let $M := \lb \bigoplus_{x \in X} M_x$ be the direct sum in $\cat{Mod} A$.
Given any collection $\{ U_x \}_{x \in X}$, where $U_x \subset M_x$ is an open 
$\k$-submodule, let $U := \bigoplus_{x \in X} U_x$, which is a $\k$-submodule 
of $M$. Put on $M$ the $\k$-linear topology generated by these $\k$-submodules 
$U$. This makes $M$ into a ST $A$-module, and together with the embeddings 
$M_x \inj M$ it becomes a coproduct in $\cat{STMod} A$. 
Likewise, the product $\prod_{x \in X} M_x$ in $\cat{Mod} A$, with the product 
topology, becomes a product in $\cat{STMod} A$. 

Let $\phi : M \to N$ be a homomorphism in $\cat{STMod} A$.
Then $\opn{Ker}(\phi)$, with the topology induced on it from $M$ (i.e.\ the 
subspace topology), is a kernel of $\phi$. The module $\opn{Coker}(\phi)$, with 
the topology induced on it from $N$ (i.e.\ the quotient topology), is a 
cokernel of $\phi$.

Now that we have coproducts, products, kernels and cokernels, any limit and 
colimit can be produced in $\cat{STMod} A$.
\end{proof}

Let $A \in \cat{STRing} \k$. We often use the notation 
$\opn{Hom}^{\mrm{cont}}_A(M, N)$
to denote the $\k$-module of continuous $A$-linear homomorphism between two 
ST left $A$-modules $M$ and $N$. This is just another way to refer to the 
$\k$-module $\opn{Hom}_{\cat{STMod} A}(M, N)$.

\begin{rem}
The concept of ST module is very close to the concept of {\em smooth 
representation} from the theory of representations of topological groups. 
Perhaps some of our work here can be used in that area.   
\end{rem}

\begin{dfn} \label{dfn:20}
Let $A$ be a ST ring, and let $M$ be a left $A$-module. The 
{\em fine $A$-module topology} on $M$ is the finest linear topology on $M$ 
that makes it into a ST $A$-module. 
\end{dfn}

It is not clear at first whether such a topology exists; but it 
does -- see \cite[Section 1.2]{Ye1}. The fine topology
can be characterized as follows: a ST 
$A$-module $M$ has the fine $A$-module topology iff for any 
$N \in \cat{STMod} A$ the canonical function 
\[ \opn{Hom}^{\mrm{cont}}_A(M, N) \to \opn{Hom}^{}_A(M, N) \]
is bijective. This is proved in \cite[Proposition 1.2.4]{Ye1}. 
(So we get a left adjoint to the forgetful functor 
$\cat{STMod} A \to \cat{Mod} A$.) 

The fine $A$-module topology can be described quite explicitly.
First consider a free module $F = \bigoplus_{x \in X} A$.
The direct sum (i.e.\ coproduct) topology on it is the fine topology. 
Now take any $A$-module $M$, and let $F \surj M$ be some 
$A$-linear surjection from a free module $F$. Then the the quotient 
topology on $M$ coincides with its fine topology.

\begin{dfn}
Let $\phi : M \to N$ be a homomorphism  in $\cat{STMod} A$.
\begin{enumerate}
\item $\phi$ is called a {\em strict monomorphism} if it is injective, and 
the topology on $M$ equals the subspace topology on it induced by $\phi$ and 
$N$. 
\item $\phi$ is called a {\em strict epimorphism} if it is surjective, and 
the topology on $N$ equals the quotient topology on it induced by $\phi$ and 
$M$.
\end{enumerate}
\end{dfn}

\begin{exa}
Let $\phi : M \to N$ be a homomorphism in $\cat{STMod} A$, and assume both 
modules have the fine $A$-module topologies. 
If $\phi$ is a surjection, then it is a strict epimorphism. 
If $\phi : M \to N$ a split injection in $\cat{STMod} A$, then it is 
a strict monomorphism. 
\end{exa}

\begin{dfn} 
Let $f : A \to B$ be a homomorphism  in $\cat{STRing} \k$.
We say that $f$ is a {\em strict monomorphism} (resp.\ {\em strict epimorphism})
in $\cat{STRing} \k$ if it is so in $\cat{STMod} \k$.
\end{dfn}

\begin{dfn} \label{dfn:21}
Let $A \in \cat{STRing}_{\mrm{c}} \k$, let $f : A \to B$ be a homomorphism in 
$\cat{Ring}_{\mrm{c}} \k$, and let $M \in \cat{Mod} B$. We view $M$ as an 
$A$-module via $f$. The fine $A$-module topology on $M$ is called 
the {\em fine $(A, f)$-module topology}.
\end{dfn}

\begin{lem} \label{lem:41}
In the situation of Definition \tup{\ref{dfn:21}:}
\begin{enumerate}
\item The ring $B$, with the fine $(A, f)$-module 
topology, becomes an object of \lb $\cat{STRing}_{\mrm{c}} \k$; and
$f : A \to B$ becomes a morphism in  $\cat{STRing}_{\mrm{c}} \k$.
\item Give $B$ the fine $(A, f)$-module topology. Then the fine $B$-module 
topology on $M$ coincides with the fine $(A, f)$-module topology on it. 
Therefore $M$, endowed with the fine $(A, f)$-module topology,
is an object of $\cat{STMod} B$.
\end{enumerate}
\end{lem}

\begin{proof}
Easy using \cite[Proposition 1.2.4]{Ye1}.
\end{proof}

\begin{lem} \label{lem:42}
Let $\{ B_i \}_{i \in \N}$ be an inverse system in 
$\cat{STRing} \k$. Then the ring
$B := \lim_{\leftarrow i} B_i$, with its inverse limit topology 
\tup{(}Proposition \tup{\ref{prop:41}(2))}, is a ST $\k$-ring. 
\end{lem}

\begin{proof}
This follows almost immediately from the definitions.
\end{proof}

Here is the most important construction of ST rings in our context. This is 
\cite[Definition 1.3.3]{Ye1}. Lemmas \ref{lem:41} and \ref{lem:42} justify it. 

\begin{dfn} \label{dfn:41}
Let $A$ be a commutative ST $\k$-ring.
\begin{enumerate}
\item Let $t$ be a variable. 
\begin{enumerate}
\item For any $i \in \N$ put on the truncated polynomial ring 
$A[t] / (t^{i + 1})$ the fine $A$-module topology. 
This makes $A[t] / (t^{i + 1})$ a ST $\k$-ring. 

\item Give the ring of formal power series
 $A[[t]] :=  \lim_{\leftarrow i} A[t] / (t^{i + 1})$
the inverse limit topology. 
In this way $A[[t]]$ is a ST $\k$-ring. 

\item Put on the ring of formal Laurent series 
$A((t)) := A[[t]][t^{-1}]$ the fine $A[[t]]$-module topology. 
In this way $A((t))$ is a ST $\k$-ring. 
\end{enumerate}
\item Let $\bsym{t} = (t_1, \ldots, t_n)$ be a sequence of variables. 
The {\em ring of iterated Laurent series} 
\[ A((\bsym{t})) = A((t_1, \ldots, t_n)) \]
is the commutative ST $\k$-ring defined recursively on $n$ by
\[ A((t_1, \ldots, t_n)) := A((t_2, \ldots, t_n))((t_1)) , \]
using part (1).
\end{enumerate}
\end{dfn}

Note that as  ST $A$-modules there is an isomorphism 
\[ A((t)) = 
\big( \prod\nolimits_{i \geq 0} A \cd t^i \big) \oplus 
\big( \bigoplus\nolimits_{i < 0} A \cd t^i \big) \cong 
\big( \prod\nolimits_{i \in \N} A \big) \oplus 
\big( \bigoplus\nolimits_{i \in \N} A \big) . \]

\begin{rem} \label{rem:40}
Strangely, for $n \geq 2$ (and when $A$ is nonzero), the ring 
$B := A((t_1, \ldots, t_n))$ is {\em not topological}; namely multiplication is 
not a continuous function $B \times B \to B$. This is the reason for 
introducing the semi-topological apparatus. 

Furthermore, the topology on $B$ is {\em not metrizable}. 
Still $B$ is {\em complete}, in the sense that the canonical homomorphism 
$B \to \lim_{\leftarrow U} B / U$, where $U$ runs over all open $\k$-submodules
of $B$, is bijective. 
\end{rem}

\begin{exer}
Let $K := \k((\bsym{t}))$, the ring of Laurent series in the sequence of 
variables $\bsym{t} = (t_1, \ldots, t_n)$, with its topology from Definition 
\ref{dfn:41}. Let $\opn{F}(\Z^n, \k)$ be the set of functions 
$a : \Z^n \to \k$, written in subscript notation, namely for 
$\bsym{i} = (i_1, \ldots, i_n) \in \Z^n$  
the value of $a$ is $a_{\bsym{i}} \in \k$. 
The notation for monomials in $\bsym{t}$ is 
$\bsym{t}^{\bsym{i}} := t_1^{i_1} \cdots t_n^{i_n}$. 
We say that a collection 
$\{ a_{\bsym{i}} \bsym{t}^{\bsym{i}} \}_{\bsym{i} \in \Z^n}$
of elements of $K$ is a {\em Cauchy collection} if for every open 
$\k$-submodule $U \subset K$ there is a finite subset $I \subset \Z^n$ such 
that $a_{\bsym{i}} \bsym{t}^{\bsym{i}} \in U$ for all $\bsym{i} \notin I$.
A function $a \in \opn{F}(\Z^n, \k)$ is called Cauchy if the collection 
$\{ a_{\bsym{i}} \bsym{t}^{\bsym{i}} \}_{\bsym{i} \in \Z^n}$
is Cauchy. The set of Cauchy functions is denoted by 
$\opn{F}_{\mrm{c}}(\Z^n, \k)$. The exercise is to show that 
for any $a \in \opn{F}_{\mrm{c}}(\Z^n, \k)$
the series 
$\sum_{\bsym{i} \in \Z^n} a_{\bsym{i}} \bsym{t}^{\bsym{i}}$
converges in $K$, and that the resulting function 
$\opn{F}_{\mrm{c}}(\Z^n, \k) \to K$ is a $\k$-linear bijection. 
(For a slightly more general assertion see the end of \cite[Section 1.3]{Ye1}.)
\end{exer}

\begin{dfn} \label{dfn:65}
Let $f : A \to B$ be a homomorphism in $\cat{STRing} \k$.
Given $M \in \cat{STMod} B$, we denote by $\opn{rest}_f(M)$ the ST $A$-module 
whose underlying ST $\k$-module is $M$, and $A$ acts via $f$. 
\end{dfn}

In this way we get a functor 
\[ \opn{rest}_f : \cat{STMod} B \to \cat{STMod} A. \]

We now return to liftings. 

\begin{dfn} \label{dfn:22}
Let $A$ be a local ring in $\cat{STRing}_{\mrm{c}} \k$, with residue 
field $K$. We put on $K$ the fine $A$-module topology, so that the canonical 
surjection $\pi : A \surj K$ is a morphism in $\cat{STRing}_{\mrm{c}} \k$.
A {\em lifting of $K$} in $\cat{STRing}_{\mrm{c}} \k$
is a homomorphism $\si : K \to A$ in $\cat{STRing}_{\mrm{c}} \k$, such that 
the composition $\pi \circ \si$ is the identity of $K$. 
\end{dfn}

The important thing to remember is that $\si : K \to A$ has to be continuous. 

\begin{exa} \label{exa:100}
Assume $\k$ is a field of characteristic $0$. 
Let $K := \k((t_2))$ and $A := K[[t_1]]$, with topologies from Definition 
\ref{dfn:41}. So we are in the situation of Definition \ref{dfn:22}.
Consider the lifting $\si : K \to A$ from Example \ref{exa:30}.
If at least one of the elements $c_i$ is nonzero, the lifting $\si$ is not 
continuous. 
\end{exa}

\begin{rem}
Let $A$ be a local ring in $\cat{STRing}_{\mrm{c}} \k$, with maximal 
ideal $\m$. We do not assume any relation between the given topology of $A$ 
and its $\m$-adic topology. For instance, $A$ could have the discrete topology, 
which is finer than any other topology. 

On the other hand, in Example \ref{exa:100} above,
where $A = \k((t_2))[[t_1]]$ and $\m = (t_1)$, the $\m$-adic topology on  $A$
is finer than the given topology on it (since the discrete topology on 
$K = \k((t_2))$ is finer than its $t_2$-adic topology). 
\end{rem}

The next definition is a generalization of \cite[Definition 2.2.1]{Ye1}.

\begin{dfn} \label{dfn:23}
Let $A$ be an artinian local ring in $\cat{STRing}_{\mrm{c}} \k$, with 
residue field $K$. We put on $K$ the fine $A$-module topology, so that the 
canonical surjection $\pi : A \surj K$ is a strict epimorphism in 
$\cat{STRing}_{\mrm{c}} \k$.
\begin{enumerate}
\item A lifting $\si : K \to A$ in $\cat{STRing}_{\mrm{c}} \k$
is called a {\em precise lifting} if 
the original topology of $A$ equals the fine $(K, \si)$-module topology on it.
\item The topology on $A$ is called a {\em precise topology}, and $A$ is called 
a {\em precise artinian local ring in $\cat{STRing}_{\mrm{c}} \k$}, if there 
exists some precise lifting $\si : K \to A$ in $\cat{STRing}_{\mrm{c}} \k$.
\end{enumerate}
\end{dfn}

Here are examples. 

\begin{exa}
Start with an artinian local ring $A$ in $\cat{Ring}_{\mrm{c}} \k$,
and with a given lifting $\si : K \to A$ of the residue field.  
Put any topology on $K$ that makes it into an object of 
$\cat{STRing}_{\mrm{c}} \k$. 
Next give $A$ the fine $(K, \si)$-module topology. 
According to Lemma \ref{lem:41}(2), the fine $A$-module topology on $K$ 
equals its original topology. We see that $\si : K \to A$ is a precise lifting, 
and hence $A$ is a  precise artinian local ring in 
$\cat{STRing}_{\mrm{c}} \k$.
\end{exa}

Definition \ref{dfn:23} makes sense also for an artinian {\em semi-local} ring 
$A$, with Jacobson radical $\r$ and residue ring $K := A / \r$. Of course here 
$K$ is a finite product of fields. This is used in the next example. 

\begin{exa} \label{exa:210}
In this example we use the Beilinson completion that's explained in Section 
\ref{sec:geom}. Assume $\k$ is a perfect field, and let $X$ be a finite type 
$\k$-scheme. Take a saturated chain of points 
$\xi = (x_0, \ldots, x_n)$ in $X$, and let 
$A := \OO_{X, x_0} / \m^{l + 1}_{x_0}$ for some $l \in \N$.
So $A$ is an artinian local ring, and its residue field is $K := \kk(x_0)$. 
Let $\si : K \to A$ be a lifting in $\cat{Ring}_{\mrm{c}} \k$. 
 
We view $A$ and $K$ as quasi-coherent sheaves on $X$, constant on  
the closed set $\ol{\{ x_0 \}}$. The lifting $\si$ is a differential operator 
of $\OO_X$-modules, and hence, according to 
\cite[Propositions 3.1.10 and 3.2.2]{Ye1}, there is a homomorphism 
$\si_{\xi} : K_{\xi} \to A_{\xi}$ 
in $\cat{STRing}_{\mrm{c}} \k$ that lifts the canonical surjection 
$\pi_{\xi} : A_{\xi} \to K_{\xi}$. 
The arguments in the proof of \cite[Proposition 3.2.5]{Ye1} show that 
$K_{\xi}$ has the fine $A_{\xi}$-module topology, and vice versa. 

By Theorem \ref{thm:50} the ring $K_{\xi}$ is a finite product of fields. 
Therefore $A_{\xi}$ is an artinian semi-local ring, with residue ring 
$K_{\xi}$. We see that the lifting $\si_{\xi} : K_{\xi} \to A_{\xi}$ is a 
precise lifting, and $A_{\xi}$ is a precise artinian semi-local ring in 
$\cat{STRing}_{\mrm{c}} \k$. 
\end{exa}

\begin{exa} \label{exa:211}
Assume $\k$ is a field. Let $K := \k((t))$ with the discrete topology, and let 
$\m := \k((t))$ with the $t$-adic topology. We view $\m$ as a ST $K$-module. 
Define $A := K \oplus \m$,
the trivial extension of $K$ by $\m$ (so $\m^2 = 0$). 
For any lifting $\si : K \to A$, the $(K, \si)$-module topology on $A$
is the discrete topology. Therefore there is no precise lifting, and 
$A$ is not a precise artinian local ring in $\cat{STRing}_{\mrm{c}} \k$.
\end{exa}

A question that immediately comes to mind is this: If $A$ is a precise artinian 
local ring in $\cat{STRing}_{\mrm{c}} \k$, are all liftings 
$\si : K \to A$ in $\cat{STRing}_{\mrm{c}} \k$ precise? This is answered 
affirmatively in Corollary \ref{cor:21} in the next section. 

Let $A$ be a ST $\k$-ring and let $M$ be a ST $A$-module. The closure 
$\ol{\{ 0 \}}$ of the zero submodule $\{ 0 \}$ is an $A$-submodule of $M$.
 
\begin{dfn} \label{dfn:25}
Let $A$ be a ST $\k$-ring and let $M$ be a ST $A$-module.
\begin{enumerate}
\item If $\{ 0 \}$ is closed in $M$, then $M$ is called a {\em separated 
ST module}.
\item Define $M^{\mrm{sep}} := M / \ol{\{ 0 \}}$. This is a ST $A$-module with 
the quotient topology from $M$, and we call it {\em the separated ST module 
associated to $M$}. 
\end{enumerate}
\end{dfn}

Of course $M$ is a separated ST module iff it is a Hausdorff topological space.

The assignment $M \mapsto M^{\mrm{sep}}$ is 
a $\k$-linear functor from $\cat{STMod} A$ to itself. 
There is a functorial strict epimorphism $\tau_M : M \to  M^{\mrm{sep}}$.
The ST module $M^{\mrm{sep}}$ is separated, and it is easy to 
see that for any separated ST $A$-module $N$ the homomorphism 
\[ \opn{Hom}_{A}^{\mrm{cont}}(M^{\mrm{sep}}, N) \to  
\opn{Hom}_{A}^{\mrm{cont}}(M, N) \]
induced by $\tau_M$ is bijective. 

\begin{rem}
The reader might wonder why we work with separated modules and not with 
complete modules. The reason is that for a ST $A$-module $M$, its completion 
$\what{M} := \lim_{\leftarrow U} M / U$, where $U$ runs over all open subgroups
of $M$, could fail to be an $A$-module! 

However, in many important instances (such as the module of differentials of a 
topological local field), the ST $A$-module $M^{\mrm{sep}}$ turns out to be 
complete. 
\end{rem}

We end this section with a discussion of ST tensor products. The next 
definition is \cite[Definition 1.2.11]{Ye1}.

\begin{dfn} \label{dfn:43}
Suppose $A$ is a commutative ST $\k$-ring, and $M_1, \ldots, M_p$ are ST 
$A$-modules. The tensor product topology on the $A$-module
\[ \bigotimes_{i = 1}^p M_i := M_1 \otimes_A \cdots \otimes_A M_p \]
is the finest linear topology such that the canonical
multilinear function 
$\prod_{i = 1}^p M_i \to \bigotimes_{i = 1}^p M_i$
is semi-continuous. 
\end{dfn}

With this topology $\bigotimes_{i = 1}^p M_i$ is a ST 
$A$-module. Given any semi-continuous $A$-multilinear function 
$\be : \prod_{i = 1}^p M_i \to N$,
where $N$ is a ST $A$-module, the corresponding $A$-linear homomorphism 
$\bigotimes_{i = 1}^p M_i \to N$ is continuous. 
For more details see \cite[Section 1.2]{Ye1}.

\begin{exa}
Let $f : A \to B$ be a homomorphism in $\cat{STRing}_{\mrm{c}} \k$, and let
$M \in \cat{STMod} A$. Then $B \ot_A M$, with the tensor product topology, is a 
ST $B$-module. We get an adjoint to the forgetful functor $\opn{rest}_f$.
If $M$ has the fine $A$-module topology, then 
$B \ot_A M$ has the fine $B$-module topology.
See \cite[Proposition 1.2.14 and Corollary 1.2.15]{Ye1}.
\end{exa}

\begin{rem}
Assume the base ring $\k$ is a field. 
Let $M$ and $N$ be ST $\k$-modules (i.e.\ linearly topologized $\k$-modules).
In \cite{Be2}, Beilinson talks about three topologies on the tensor product 
$M \ot_{\k} N$.
 
In our paper we encounter two topologies on $M \ot_{\k} N$.
The first is the tensor product topology from Definition \ref{dfn:43}. 
It is symmetric: an isomorphism $M \cong N$ in $\cat{STMod} \k$
induces an isomorphism 
$M \ot_{\k} N \cong N \ot_{\k} M$ in $\cat{STMod} \k$.

For the second kind of tensor product topology consider 
$M := \k(t_2)$ with the $t_2$-adic topology, and $N := \k(t_1)$
the $t_1$-adic topology. So $M \cong N$ in $\cat{STMod} \k$.
Let $K := \k((t_1, t_2))$ be the field of iterated Laurent series, with the
topology of Definition \ref{dfn:41}, starting from the discrete topology 
on $\k$. The embedding $M \ot_{\k} N \subset K$
induces a topology on it, making it into a ST $\k$-module.
Presumably this topology on $M \ot_{\k} N$ can be described in terms of  
the topologies of $M$ and $N$. 
Now $K$ is complete, and $M \ot_{\k} N$ is dense in it. Since the roles of the 
two variables in $K$ is different (e.g.\ the series 
$\sum_{i \in \N} t_1^i \cd t_2^{-i}$ is summable, but the series 
$\sum_{i \in \N} t_1^{-i} \cd t_2^i$ is not summable),
we see that this topology on $M \ot_{\k} N$ is not symmetric. 

It should be interesting to compare our two tensor product topologies to the 
three discussed in \cite{Be2}.
\end{rem}

%\cleardoublepage
\section{Continuous Differential Operators}
\label{sec:CDOs}

Our approach to continuous differential operators is an adaptation to the ST 
context of the definitions from \cite{EGA-IV}.
We are following \cite{Ye1} and \cite{Ye2}.
Recall that the base ring $\k$ is a nonzero commutative ring, and it has  
the discrete topology.

Let $A$ be a commutative $\k$-ring. 
Any $\k$-central $A$-bimodule $P$ has an increasing filtration
$\{ F_i(P) \}_{i \in \Z}$ by $A$-sub-bimodules, called the differential 
filtration. This filtration is defined inductively.
For $i \leq -1$ we define $F_i(P) := 0$. For $i \geq 0$ the elements of 
$F_i(P)$ are the elements $p \in P$ such that 
$a \cd p - p \cd a \in F_{i - 1}(P)$
for every $a \in A$. 

Now assume $A$ is a commutative ST $\k$-ring, and let $M, N$ be ST 
$A$-modules. The set $\opn{Hom}_{\k}^{\mrm{cont}}(M, N)$ of continuous 
$\k$-linear homomorphisms is a $\k$-central $A$-bimodule, so it has a 
differential filtration. We define 
\[ \opn{Diff}_{A / \k}^{\mrm{cont}}(M, N) := 
\bigcup_{i} \ F_i \bigl( \opn{Hom}_{\k}^{\mrm{cont}}(M, N) \bigr) 
\subset \opn{Hom}_{\k}^{\mrm{cont}}(M, N) . \]
The elements of 
\[  F_i \bigl( \opn{Diff}_{A / \k}^{\mrm{cont}}(M, N) \bigr) := 
F_i \bigl( \opn{Hom}_{\k}^{\mrm{cont}}(M, N) \bigr) \]
are by definition continuous differential operators of order $\leq i$. 
Note that 
\[ F_0 \bigl( \opn{Diff}_{A / \k}^{\mrm{cont}}(M, N) \bigr)  =
\opn{Hom}_{A}^{\mrm{cont}}(M, N) . \]

When $N = M$ we write 
\[ \opn{Diff}_{A / \k}^{\mrm{cont}}(M) :=  
\opn{Diff}_{A / \k}^{\mrm{cont}}(M, M) . \]
This is a subring of $\opn{End}_{\k}^{\mrm{cont}}(M)$.
Let $\opn{Der}_{A / \k}^{\mrm{cont}}(M)$ be the $A$-module of continuous 
$\k$-linear derivations $A \to M$. Then 
\[ F_1 \bigl( \opn{Diff}_{A / \k}^{\mrm{cont}}(A, M) \bigr)  =
M \oplus \opn{Der}_{A / \k}^{\mrm{cont}}(M) \]
as left $A$-modules. 

If $M = A$ then we write 
\begin{equation} \label{eqn:25}
\DD^{\mrm{cont}}_{A / \k} := \opn{Diff}_{A / \k}^{\mrm{cont}}(A) .
\end{equation}
This is the ring of continuous differential operators of $A$ (relative to 
$\k$). Let us write 
$\mcal{T}^{\mrm{cont}}_{A / \k} := \opn{Der}_{A / \k}^{\mrm{cont}}(A)$,
the Lie algebra of continuous derivations of $A$. Then  
\[ F_1(\DD^{\mrm{cont}}_{A / \k}) = A \oplus \mcal{T}^{\mrm{cont}}_{A / \k} \]
as left $A$-modules. 

If $A$ is discrete, then 
$\DD^{\mrm{cont}}_{A / \k} = \DD_{A / \k}$, the usual ring of differential 
operators from \cite{EGA-IV}; and 
$\mcal{T}^{\mrm{cont}}_{A / \k} = \mcal{T}_{A / \k}$, the usual Lie algebra of 
derivations.

\begin{rem}
There is a canonical topology on $\opn{Hom}_{\k}^{\mrm{cont}}(M, N)$, called 
the {\em Hom topology}, making it a ST $A$-module. See 
\cite[Definition 1.1]{Ye2}. However, in this paper we shall not need this 
topology, and hence we consider $\opn{Hom}_{\k}^{\mrm{cont}}(M, N)$ as an 
untopologized object (or as a discrete ST $\k$-module). 
\end{rem}

\begin{exa} \label{exa:68}
Let $\bsym{t} = (t_1, \ldots, t_n)$ be a sequence of variables of length 
$n \geq 1$. In Definition \ref{dfn:41} we saw how to make the ring of iterated 
Laurent series 
$\k((\bsym{t})) := \k((t_1, \ldots, t_n))$
into a ST $\k$-ring. This a separated ST ring, i.e.\ 
$\k((\bsym{t})) = \k((\bsym{t}))^{\mrm{sep}}$. Let 
$\k[\bsym{t}]$ be the polynomial ring, with discrete topology. 
According to \cite[Corollary 1.5.19]{Ye1} the ring homomorphism 
$\k[\bsym{t}] \to \k((\bsym{t}))$ is {\em topologically \'etale relative to 
$\k$}. This implies that any $\k$-linear differential operator $\phi$ on 
$\k[\bsym{t}]$ extends uniquely to a continuous $\k$-linear differential 
operator $\what{\phi}$ on $\k((\bsym{t}))$. This gives us a ring homomorphism 
$\DD_{\k[\bsym{t}] / \k} \to \DD^{\mrm{cont}}_{\k((\bsym{t})) / \k}$
that respects the differential filtrations, and such that the induced 
homomorphism
\begin{equation} \label{eqn:26}
\k((\bsym{t})) \ot_{\k[\bsym{t}]} \DD_{\k[\bsym{t}] / \k} \to 
\DD^{\mrm{cont}}_{\k((\bsym{t})) / \k}
\end{equation}
is bijective. 

If $\k$ has characteristic $0$ (i.e.\ $\Q \subset \k$), then by (\ref{eqn:26}) 
any $\what{\phi} \in F_l(\DD^{\mrm{cont}}_{\k((\bsym{t})) / \k})$
can be expressed uniquely as a finite sum
\begin{equation} \label{eqn:60}
\what{\phi} = 
\sum_{(i_1, \ldots, i_n)} a_{(i_1, \ldots, i_n)} \cd 
\pa_1^{i_1} \cdots \pa_n^{i_n} ,
\end{equation}
where $i_k \in \N$, $\sum_k i_k \leq l$, 
$a_{(i_1, \ldots, i_n)} \in \k((\bsym{t}))$ and 
$\partial_i := \frac{\partial}{\partial t_i}$.

On the other hand, if $\k$ has characteristic $p > 0$ (i.e.\ $\F_p \subset 
\k$), 
then the structure of $\DD^{\mrm{cont}}_{\k((\bsym{t})) / \k}$ is totally 
different. For every $m \geq 0$ let 
$\k((\bsym{t}^{p^m})) := \k((t_1^{p^m}, \ldots, t_n^{p^m}))$, 
which is a subring of $\k((\bsym{t}))$. The ring $\k((\bsym{t}))$ is a free 
module over $\k((\bsym{t}^{p^m}))$ of rank $p^{n m}$, and the topology on 
$\k((\bsym{t}))$ is the fine $\k((\bsym{t}^{p^m}))$-module topology. 
According to \cite[Theorem 1.4.9 and Corollary 2.1.18]{Ye1} we have 
\[ \DD^{\mrm{cont}}_{\k((\bsym{t})) / \k} = 
\DD_{\k((\bsym{t})) / \k} = 
\bigcup_{m \geq 0} 
\opn{End}_{\k((\bsym{t}^{p^m}))} \bigl( \k((\bsym{t})) \bigr) . \]
\end{exa}

Let $B$ be a $\k$-ring (not necessarily commutative). For any $r_1, r_2 
\in \N$ let \lb $\opn{Mat}_{r_2 \times r_1}(B)$
be the set of $r_2 \times r_1$ matrices with entries in $B$.
The set of matrices 
$\opn{Mat}_{r}(B) := \opn{Mat}_{r \times r}(B)$ 
is a $\k$-ring with matrix multiplication,
and $\opn{Mat}_{r_2 \times r_1}(B)$ is a $\k$-central 
$\opn{Mat}_{r_2}(B)$-$\opn{Mat}_{r_1}(B)$-bimodule.
The group of invertible elements of $\opn{Mat}_{r}(B)$ is denoted by 
$\opn{GL}_{r}(B)$.

Now consider some $M \in \cat{Mod} \k$.
The $\k$-ring $B := \opn{End}_{\k}(M)$ acts on $M$ from the left. 
We view $M^{r_1}$ as a column module, namely we make the identification
$M^{r_1} = \opn{Mat}_{r_1 \times 1}(M)$. Then for any 
$\phi \in \opn{Mat}_{r_2 \times r_1} (B)$
and $m \in M^{r_1}$, the matrix product $\phi \cd m$ is an element 
of $M^{r_2}$.
In this way we obtain a canonical isomorphism 
\begin{equation} \label{eqn:27}
\opn{Hom}_{\k}(M^{r_1}, M^{r_2}) \cong 
\opn{Mat}_{r_2 \times r_1} \bigl( \opn{End}_{\k}(M) \bigr) 
= \opn{Mat}_{r_2 \times r_1} (B)
\end{equation}
of left $\opn{Mat}_{r_2}(B)$-modules and right 
$\opn{Mat}_{r_1}(B)$-modules.

The next lemma shows that this also happens in the topological and differential
contexts.

\begin{lem} \label{lem:20}
Let $A \in \cat{STRing}_{\mrm{c}} \k$ and 
$M \in \cat{STMod} A$. For any natural numbers $r_1$ and $r_2$, 
matrix multiplication gives rise to bijections 
\[ \opn{Mat}_{r_2 \times r_1} \bigl( \opn{End}^{\mrm{cont}}_{\k}(M) \bigr) 
\cong \opn{Hom}^{\mrm{cont}}_{\k}(M^{r_1}, M^{r_2}) \]
and
\[ \opn{Mat}_{r_2 \times r_1} 
\bigl( \opn{Diff}^{\mrm{cont}}_{A / \k}(M) \bigr) 
\cong \opn{Diff}^{\mrm{cont}}_{A / \k}(M^{r_1}, M^{r_2}) . \]
In particular, a homomorphism 
$\phi : M^r \to M^r$ in $\cat{STMod} \k$ is an isomorphism iff the 
corresponding matrix belongs to 
$\opn{GL}_{r} \bigl( \opn{End}^{\mrm{cont}}_{\k}(M) \bigr)$.
\end{lem}

\begin{proof}
This is a straightforward consequence of the definitions. 
\end{proof}

Lifting, precise liftings and precise artinian local rings in 
$\cat{STRing}_{\mrm{c}} \k$ were introduced in Definitions \ref{dfn:22} and 
\ref{dfn:23}. The main result of this section is the next theorem.

\begin{thm} \label{thm:20}
Let $A$ be a precise artinian local ring in $\cat{STRing}_{\mrm{c}} \k$, with 
residue field $K$. Put on $K$ the fine $A$-module topology. 
Let $\si_1, \si_2 : K \to A$ be liftings in $\cat{STRing}_{\mrm{c}} \k$
of the canonical surjection $A \surj K$, and assume that $\si_2$ is a precise 
lifting. 

Let $M_1$ and $M_2$ be finite $A$-modules, and let $\phi : M_1 \to M_2$ be an 
$A$-linear homomorphism. For $l = 1, 2$ choose $K$-linear isomorphisms 
$\psi_l : K^{r_l} \iso \opn{rest}_{\si_l}(M_l)$.
Let 
$\bar{\phi} \in \opn{Mat}_{r_2 \times r_1} \bigl( \opn{End}_{\k}(K) \bigr)$
be the matrix such that the diagram 
\[ \UseTips \xymatrix @C=6ex @R=6ex {
M_1
\ar[r]^{\phi}
&
M_2
\\
K^{r_1}
\ar[u]^{\psi_1}
\ar[r]^{\bar{\phi}}
&
K^{r_2}
\ar[u]_{\psi_2}
} \]
in $\cat{Mod} \k$ is commutative. Then the following hold. 

\begin{enumerate}
\item The matrix $\bar{\phi}$ belongs to 
$\opn{Mat}_{r_2 \times r_1}(\DD^{\mrm{cont}}_{K / \k})$.

\item Assume that $M_1 = M_2$ and $\phi$ is the identity isomorphism. 
Write $r := r_1$. 
Then the matrix $\bar{\phi}$ belongs to 
$\opn{GL}_{r}(\DD^{\mrm{cont}}_{K / \k})$.
\end{enumerate}
\end{thm}

\begin{proof}
(1) Put on $M_1$ and $M_2$ the fine $A$-module topologies. 
Let us write $\bar{M}_l := K^{r_l}$; these are ST $K$-modules with the fine 
$K$-module topologies. 
Since both $\si_1, \si_2 : K \to A$ are continuous, it follows that both
$\psi_l : \bar{M}_l \to M_l$ 
are continuous, namely are homomorphisms in 
$\cat{STMod} \k$. Furthermore, because $\si_2$ is a precise lifting, 
it follows that 
$\psi_2 : \bar{M}_2 \to M_2$ 
is a homeomorphism, so it is an 
isomorphism in $\cat{STMod} \k$.
We conclude that 
$\bar{\phi} = \psi_2^{-1} \circ \phi \circ \psi_1$
is a homomorphism in $\cat{STMod} \k$, namely it is continuous.

Next, let us view $A$ as a $K$-ring via $\si_1$. (No topology in this 
paragraph.) The canonical surjection 
$A \surj K$ makes $A$ into an augmented $K$-ring. Let us view $\bar{M}_1$ as 
an $A$-module via this augmentation. Now both $\bar{M}_1$ and $M_1$ are finite 
length $A$-modules, and $\psi_1 : \bar{M}_1 \to  M_1$ is $K$-linear. According 
to \cite[Proposition 1.4.4]{Ye1}, $\psi_1$ is a differential operator over $A$. 
(The order of this operator is bounded by $r_1 - 1$.)

Similarly, we can view $A$ as an augmented $K$-ring via $\si_2$.
(No topology in this paragraph either.) 
Now both $\psi_2 : \bar{M}_2 \to  M_2$ and its inverse 
$\psi_2^{-1} :  M_2 \to \bar{M}_2$ are $K$-linear, and therefore they are 
differential operators over $A$. We conclude that the composition 
\[ \bar{\phi} = \psi_2^{-1} \circ \phi \circ \psi_1 : \bar{M}_1 \to \bar{M}_2 \]
is a differential operator over $A$. Here the liftings $\si_1, \si_2$ stop 
playing a role. Since $A$ acts on $\bar{M}_1$ and $\bar{M}_2$ via 
the canonical surjection $A \surj K$, and this implies that 
$\bar{\phi}$ is a differential operator over $K$.

Combining the two results above we conclude that 
$\bar{\phi} : \bar{M}_1 \to \bar{M}_2$ is a continuous differential operator 
over $K$. Using Lemma \ref{lem:20} we see that the matrix $\bar{\phi}$ belongs 
to $\opn{Mat}_{r_2 \times r_1} (\DD^{\mrm{cont}}_{K / \k})$.
This establishes (1). 

\medskip \noindent
(2) The proof of this part is very similar to that of \cite[Lemma 6.6]{Ye2} 

Let $\m$ be the maximal ideal of $A$, and write $M := M_1$. 
Consider the $\m$-adic filtration on $M$. The associated graded module 
$\opn{gr}_{\m}(M)$ is a $K$-module of length $r$ (regardless of any lifting). 
By a {\em filtered $K$-basis} of $M$ we mean a collection 
$\{ m_i \}_{1 = 1}^r$
of elements of $M$, such that the collection of symbols 
$\{ \bar{m}_i \}_{1 = 1}^r$ is a $K$-basis of $\opn{gr}_{\m}(M)$, and 
such that $\opn{deg}(\bar{m}_{i}) \leq \opn{deg}(\bar{m}_{i + 1})$. Such bases 
exist; simply choose a graded basis of $\opn{gr}_{\m}(M)$, suitably 
ordered, and lift it to $M$. 

Choose a filtered $K$-basis $\{ m_i \}_{1 = 1}^r$ of $M$.
For $l = 1, 2$ let $\chi_l : \bar{M} \to \opn{rest}_{\si_l}(M)$ be the 
$K$-linear isomorphism corresponding to this filtered basis. We get 
a commutative diagram 
\[ \UseTips \xymatrix @C=8ex @R=6ex {
M
\ar[r]^{\phi = \bsym{1}_M}
&
M
\ar[r]^{\phi = \bsym{1}_M}
&
M
\ar[r]^{\phi = \bsym{1}_M}
&
M
\\
\bar{M} 
\ar[u]^{\psi_1}
\ar[r]^{\bar{\phi}_1}
\ar@(d,d)[rrr]^{\bar{\phi}}
&
\bar{M}
\ar[u]_{\chi_1}
\ar[r]^{\bar{\phi}'}
&
\bar{M} 
\ar[u]_{\chi_2}
\ar[r]^{\bar{\phi}_2}
&
\bar{M} 
\ar[u]_{\psi_2}
} \]
in $\cat{Mod} \k$. By what we already know from (1), the matrices in the bottom 
row belong to $\opn{Mat}_r(\DD^{\mrm{cont}}_{K / \k})$; 
and  they satisfy 
$\bar{\phi} = \bar{\phi}_2 \circ \bar{\phi}' \circ \bar{\phi}_1$.
Moreover $\bar{\phi}_1, \bar{\phi}_2 \in \opn{GL}_r(K) \subset 
\opn{GL}_r(\DD^{\mrm{cont}}_{K / \k})$.
Thus it suffices to prove that 
$\bar{\phi}' \in \opn{GL}_r(\DD^{\mrm{cont}}_{K / \k})$.

Write 
$\bar{\phi}' = [\ga_{i, j}]$ with $\ga_{i, j} \in \DD^{\mrm{cont}}_{K / \k}$. 
These operators satisfy 
\begin{equation} \label{eqn:44}
\sum_{i = 1}^r \si_1(a_i) \cd m_i = 
\sum_{i, j = 1}^r \si_2( \ga_{i, j}(a_i)) \cd m_j 
\end{equation}
for any column $[a_i] \in K^r$. Therefore, for any $i$ and any $a \in K$,
taking $a_i := a$ and $a_j := 0$ for $j \neq i$, formula (\ref{eqn:44}) gives
\[ \si_1(a) \cd m_i = 
\sum_{j = 1}^r \si_2( \ga_{i, j}(a)) \cd m_j . \]
But the basis $\{ m_i \}_{1 = 1}^r$ is filtered, and this implies 
that $\ga_{i, j}(a) = 0$ for $j < i$ and $\ga_{i, i}(a) = a$. 
As elements of the ring $\DD^{\mrm{cont}}_{K / \k}$ we get 
$\ga_{i, j} = 0$ for $j < i$ and $\ga_{i, i} = 1$. 
So the matrix $\bar{\phi}'$ is upper triangular with $1$ on the diagonal: 
\[ \bar{\phi}' = [\ga_{i, j}] =
\bmat{
1 & * & \cdots & *
\\
0 & 1  & \cdots & * 
% \\
% 0 & 0 & 1 & * & \cdots & * 
\\
\vdots & \vdots & \ddots & \vdots 
\\
0 & 0 &  \cdots & 1
} 
\in \opn{Mat}_r(\DD^{\mrm{cont}}_{K / \k}) . 
\]
The matrix 
$\ep := 1 - \bar{\phi}'  \in \opn{Mat}_r(\DD^{\mrm{cont}}_{K / \k})$
is nilpotent, and hence the matrix 
\[ \th := \sum_{i = 0}^r \ep^i \in \opn{Mat}_r(\DD^{\mrm{cont}}_{K / \k}) \]
satisfies $\th \circ \bar{\phi}' = \bar{\phi}' \circ \th = 1$.
We conclude that 
$\bar{\phi}' \in \opn{GL}_r(\DD^{\mrm{cont}}_{K / \k})$.
\end{proof}

\begin{rem}
An attempt to deduce assertion (2) of the theorem from assertion (1) by 
functoriality will not work. This is because (a priori) there is no symmetry 
between the two liftings $\si_1$ and $\si_2$: only the lifting $\si_2$ is 
assumed to be precise. 

Eventually we know (Corollary \ref{cor:21}) that the lifting $\si_1$ is 
also precise. But this relies on Theorem \ref{thm:20}~!
\end{rem}

\begin{cor} \label{cor:20}
In the situation of part \tup{(2)} of the theorem, the fine $(K, \si_1)$-module 
topology on $M_1 = M_2$ equals the fine $(K, \si_2)$-module topology on it.
\end{cor}

\begin{proof}
For $l = 1, 2$ let us denote by $M^{\mrm{st}}_l$ the $\k$-module $M_l$ endowed 
with the fine $(K, \si_l)$-module topology. We want to prove that 
$M^{\mrm{st}}_1 = M^{\mrm{st}}_2$; or equivalently, 
we want to prove that the identity automorphism $\phi : M_1 \to M_2$
in $\cat{Mod} \k$ becomes an isomorphism 
$\phi^{\mrm{st}} : M^{\mrm{st}}_1 \to M^{\mrm{st}}_2$  in $\cat{STMod} \k$.

We have equality
$\phi = \psi_2 \circ \bar{\phi} \circ \psi_1^{-1}$
of isomorphisms $M_1 \to M_2$ in  $\cat{Mod} \k$. By definition of 
the fine topology, $\psi_l : K^{r_l} \iso M^{\mrm{st}}_l$ are isomorphisms 
in $\cat{STMod} \k$. Therefore it suffices to prove that 
$\bar{\phi} : K^{r_1} \to K^{r_2}$ is an isomorphism in 
$\cat{STMod} \k$.
This is true by Lemma \ref{lem:20} and  part (2) of the theorem.
\end{proof}

The next corollary is a generalization of \cite[Proposition 2.2.2(a)]{Ye1}.

\begin{cor} \label{cor:21}
Let $A$ be a precise artinian local ring in $\cat{STRing}_{\mrm{c}} \k$,
with residue field $K$. Give $K$ the fine $A$-module topology.
Then any lifting $\si : K \to A$ in $\cat{STRing}_{\mrm{c}} \k$ is a precise 
lifting.
\end{cor}

\begin{proof}
Write $\si_1 := \si$. 
By definition there exists some precise lifting 
$\si_2 : K \to A$; so the topology on $A$ equals the fine $(K, \si_2)$-module 
topology. Now apply Corollary \ref{cor:20} with $M := A$.
\end{proof}

Here is another corollary, pointed out to us by Wolfson. 

\begin{cor} 
Let $A$ be a precise artinian local ring in $\cat{STRing}_{\mrm{c}} \k$,
with residue field $K$, and let $\si : K \to A$ be a precise lifting. 
Let $M$ be a finite $A$-module, and choose a $K$-linear isomorphism
$K^{r} \iso \opn{rest}_{\si}(M)$. Then $A$ acts on $M \cong K^r$ via 
$\opn{Mat}_r(\DD^{\mrm{cont}}_{K / \k})$.
\end{cor}

\begin{proof}
In the theorem take $M_l := M$. For $a \in A$ we get an $A$-linear homomorphism
$\phi : M \to M$, $\phi(m) := a \cd m$. 
\end{proof}

\begin{rem}
If we only wanted to know that 
\[ \bar{\phi} \in 
\opn{Mat}_{r_2 \times r_1} \bigl( \opn{End}^{\mrm{cont}}_{\k}(K) \bigr)  \]
in Theorem \ref{thm:20}(1), and that in part (2) 
\[ \bar{\phi} \in 
\opn{GL}_{r} \bigl( \opn{End}^{\mrm{cont}}_{\k}(K) \bigr) , \]
then we did not have to talk about differential operators at all, and the proof 
could have been included in Section \ref{sec:ST}. The reason for placing the 
proof here is twofold. First, it is more economical to prove the full result at 
once. 

The second reason is more delicate. Sometimes in 
characteristic $p > 0$, differential operators are automatically continuous. 
See Example \ref{exa:68}. In such cases all liftings 
$\si : K \to A$ are continuous. This says that Theorem \ref{thm:20} could hold 
without assuming a priori that the liftings $\si_1, \si_2 : K \to A$ are 
continuous. 
\end{rem}

We finish this section with a discussion of differential forms.
This will be needed in Section \ref{sec:resid}. 
Recall that for $A \in \cat{Ring}_{\mrm{c}} \k$ we have the de Rham complex, 
or the DG ring of K\"ahler differentials, 
$\Om_{A / \k} = \bigoplus_{i \geq 0} \Om^i_{A / \k}$, with its differential 
$\d$. In degree $0$ we have $\Om^0_{A / \k} = A$, and the $\k$-linear 
derivation $\d : A \to \Om^1_{A / \k}$ is universal, in the sense that for any 
$\k$-linear derivation $\pa : A \to M$ there is a unique $A$-linear 
homomorphism $\phi : \Om^1_{A / \k} \to M$ such that $\pa = \phi \circ \d$. 
The $A$-module $\Om^i_{A / \k}$ is the $i$-th exterior power of the $A$-module
$\Om^1_{A / \k}$, and the operator $\d$ on $\Om_{A / \k}$ is the unique 
extension of $\d : \Om^0_{A / \k} \to \Om^1_{A / \k}$ to an odd derivation. 

Now consider $A \in \cat{STRing}_{\mrm{c}} \k$. 
The abstract DG ring $\Om_{A / \k}$ is too big (at least in characteristic 
$0$). However the DG ring $\Om_{A / \k}$ has a canonical ST structure.
For every $i$ consider the $(i + 1)$-st tensor power 
$\opn{T}^{i + 1}_{\k}(A) := A \ot_{\k} \cdots \ot_{\k} A$, with 
its tensor product topology (Definition \ref{dfn:43}). 
There is a surjection 
$\opn{T}^{i + 1}_{\k}(A) \surj \Om^i_{A / \k}$
sending
\[ a_0 \ot a_1 \ot \cdots \ot a_i \mapsto a_0 \cd \d(a_1) \cdots \d(a_i) , \]
and we use it to give $\Om^i_{A / \k}$ the quotient topology. 
Then $\Om_{A / \k} = \bigoplus_{i \geq 0} \Om^i_{A / \k}$
gets the direct sum topology. It turns out that $\Om_{A / \k}$ becomes a DG ST 
ring. In particular the differential $\d$ is continuous.
For any $i$ let 
$\Om^{i, \mrm{sep}}_{A / \k} := (\Om^{i}_{A / \k})^{\mrm{sep}}$,
the associated separated ST module.
And let 
$\Om^{\mrm{sep}}_{A / \k} := \bigoplus_{i \geq 0} \Om^{i, \mrm{sep}}_{A / \k}$, 
with the direct sum topology. Note that 
$\Om^{0, \mrm{sep}}_{A / \k} = A^{\mrm{sep}}$.

\begin{prop}[\cite{Ye1}] \label{prop:68}
Let $A$ be a commutative ST $\k$-ring. 
\begin{enumerate}
\item The ST $\k$-module $\Om^{\mrm{sep}}_{A / \k}$ has a structure of DG ST 
$\k$-ring, such the canonical surjection 
$\tau_A : \Om^{}_{A / \k} \surj \Om^{\mrm{sep}}_{A / \k}$ 
is a homomorphism of DG ST $\k$-rings.

\item Let $M$ be a separated ST $A$-module. The derivation 
$\d : A \to \Om^{1, \mrm{sep}}_{A / \k}$ induces a bijection 
\[ \opn{Hom}_{A}^{\mrm{cont}}(\Om^{1, \mrm{sep}}_{A / \k}, M) \iso 
\opn{Der}_{A / \k}^{\mrm{cont}}(M) . \]
\end{enumerate}
\end{prop}

For a proof and full details see \cite[Section 1.5]{Ye1}. 

\begin{exa} \label{exa:69}
Let $K := \k((t_1, \ldots, t_n))$ be as in Example \ref{exa:68}, 
and let $\k[\bsym{t}] := \k[t_1, \ldots, t_n]$. 
Since $K$ is a separated ST $\k$-ring, we see that 
$\Om^{0, \mrm{sep}}_{K / \k} = K$.
Because the homomorphism $\k[\bsym{t}] \to K$ is topologically 
\'etale in $\cat{STRing}_{\mrm{c}} \k$, it follows that 
$\Om^{1, \mrm{sep}}_{K / \k}$ is a free $K$-module with basis 
the sequence $\bigl( \d(t_1), \ldots, \d(t_n) \bigr)$.
For every $i$ we have 
\[ \Om^{i, \mrm{sep}}_{K / \k} = 
\bigwedge\nolimits^i_K  \Om^{1, \mrm{sep}}_{K / \k} , \]
a free $K$-module of rank $\binom{n}{i}$, with the fine $K$-module 
topology. For proofs see \cite[Corollaries 1.5.19 and 1.5.13]{Ye1}.

Note that if $\k$ is a field of characteristic $0$, then the $K$-module 
$\Om^{1}_{K / \k}$ is a free $K$-module of rank equal to 
$\opn{tr{.}deg}_{\k}(K)$, which is is uncountably infinite. Thus the kernel of 
the canonical surjection 
$\tau_K : \Om^{1}_{K / \k} \surj \Om^{1, \mrm{sep}}_{K / \k}$
is gigantic. 
\end{exa}

%\cleardoublepage
\section{Topological Local Fields}

In this section we review definitions and results from \cite[Section 2.1]{Ye1}.
We start with a definition due to Parshin \cite{Pa1, Pa3} and Kato \cite{Ka}. 
See also \cite{HLF}.

\begin{dfn} \label{dfn:10}
Let $K$ be a field. An {\em $n$-dimensional local field structure}
on $K$, for $n \geq 1$, is a sequence 
$\OO_1(K), \ldots, \OO_n(K)$ of complete discrete valuation rings, such that:
\begin{itemize}
\item $K$ is the fraction field of $\OO_1(K)$.
\item For $1 \leq i \leq n - 1$, the residue field 
of $\OO_i(K)$ is the fraction field of $\OO_{i + 1}(K)$.
\end{itemize}

The data $\bigl( K, \{ \OO_i(K) \}_{i = 1}^n \bigr)$
is an {\em $n$-dimensional local field}. We refer to $\OO_i(K)$ as the $i$-th 
valuation ring of $K$. The residue field of $\OO_i(K)$ is denoted by $\kk_i(K)$,
and its maximal ideal is denoted by $\m_i(K)$. We also write $\kk_0(K) := K$.
\end{dfn}

Let $K$ be an $n$-dimensional local field. A {\em system of uniformizers} in 
$K$ (called a regular system of parameters in \cite{Ye1})
is a sequence $(a_1, \ldots, a_n)$ of elements of $\OO_1(K)$, such 
that $a_1$ generates the maximal ideal $\m_1(K)$ of $\OO_1(K)$, and if 
$n \geq 2$, the sequence $(\bar{a}_2, \ldots, \bar{a}_n)$,
which is the image of $(a_2, \ldots, a_n)$ under the canonical surjection 
$\OO_1(K) \surj \kk_1(K)$, is a system of uniformizers in $\kk_1(K)$.
A system of uniformizers $\bsym{a} = (a_1, \ldots, a_n)$ in $K$ determines a 
valuation on $K$, with values in the group $\Z^n$ ordered 
lexicographically. 

It is easy to find a system of uniformizers in an $n$-dimensional local field 
$K$. Say $(\bar{a}_2, \ldots, \bar{a}_n)$ is a system of uniformizers in 
$\kk_1(K)$. Choose an arbitrary lifting to a sequence 
$(a_2, \ldots, a_n)$ in $\OO_1(K)$, and append to it any uniformizer $a_1$ of 
$\OO_1(K)$.

Let $\OO(K)$ be the subring of $K$ defined by
\begin{equation}
\OO(K) := \OO_1(K) \times_{\kk_1(K)} \OO_2(K) \cdots  
\times_{\kk_{n - 1}(K)} \OO_n(K) .
\end{equation}
This is a local ring, whose residue field is $\kk_n(K)$. 
We call $\OO(K)$ the ring of integers of $K$. 
The ring $\OO(K)$ is integrally closed in its field of fractions $K$; but 
unless $n = 1$ (in which case $\OO(K) = \OO_1(K)$), $\OO(K)$ is not 
noetherian.

A $0$-dimensional local field is just a field; there are no valuations. Its 
ring of integers is $\OO(K) := K$, and $\kk_0(K) := K$ too. 
 
\begin{dfn} \label{dfn:11}
Let $\k$ be a nonzero commutative ring. An {\em $n$-dimensional local field 
over $\k$}, for $n \in \N$, is an $n$-dimensional local field 
$\bigl( K, \{ \OO_i(K) \}_{i = 1}^n \bigr)$,
together with a ring homomorphism $\k \to \OO(K)$. 
The condition is that the induced ring homomorphism $\k \to \kk_n(K)$ is 
finite.
\end{dfn}

In other words, the $n$-dimensional local field structure of $K$ lives in the 
category \lb $\cat{Ring}_{\mrm{c}} \k$ of commutative $\k$-rings.
In case $\k$ is a field, then $\kk_n(K)$ is a finite field extension of $\k$.

By abuse of notation, we usually call $K$ an $n$-dimensional local 
field over $\k$, and keep the data $\{ \OO_i(K) \}_{i = 1}^n$ implicit. 

\begin{rem}
Some authors insist that the base ring is $\k = \Z$; this forces $\kk_n(K)$ to 
be a finite field. We do not impose such a restriction. 
\end{rem}

\begin{dfn} \label{dfn:15}
Let $K$ and $L$ be $n$-dimensional local fields over $\k$, for $n \geq 0$. A 
{\em morphism of $n$-dimensional local fields over $\k$} is a $\k$-ring 
homomorphism $f : K \to L$ such that the following conditions hold when 
$n \geq 1$: 
\begin{itemize}
\item $f(\OO_1(K)) \subset \OO_1(L)$.
\item The induced $\k$-ring homomorphism $f : \OO_1(K) \to \OO_1(L)$ is a 
local homomorphism.
\item The induced $\k$-ring homomorphism 
$\bar{f} : \kk_1(K) \to \kk_1(L)$ is  a morphism of $(n - 1)$-dimensional 
local fields over $\k$.
\end{itemize}
\end{dfn}

The category of $n$-dimensional local fields over $\k$ is denoted by 
$\cat{LF}^n \k$. Note that any morphism in $\cat{LF}^n \k$ is finite. 
Cf.\ Remark \ref{rem:30} below regarding more general morphisms between local 
fields. 

\begin{rem}
It can be shown that a field $K$ in $\cat{Ring}_{\mrm{c}} \k$ admits at 
most one structure of $n$-dimensional local field (see e.g.\ \cite[Remark 
2.3]{Mo}).
This implies that the forgetful functor 
$\cat{LF}^n  \k \to \cat{Ring}_{\mrm{c}} \k$ 
is fully faithful. 
\end{rem}

{}From here on we assume that the base ring $\k$ is a {\em perfect field}.
This implies that all our local fields are of equal characteristics. 

\begin{dfn} \label{dfn:35}
Let $\k$ be a perfect field. Given a finite field extension $\k'$ of $\k$, the 
{\em standard $n$-dimensional topological local field} over $\k$ with last 
residue field $\k'$ is the field 
of iterated Laurent series 
\[ \k' ((t_1, \ldots, t_n)) := \k' ((t_n)) \cdots ((t_1)) . \]
Let us write $K := \k' ((t_1, \ldots, t_n))$.
The field $K$ comes equipped with these two structures:
\begin{enumerate}
\item A structure of $n$-dimensional local field, in which the valuation rings 
are 
\[ \OO_i(K) := \k'((t_{i + 1}, \ldots, t_n))[[t_i]] , \]
and the residue fields are 
\[ \kk_i(K) := \k'((t_{i + 1}, \ldots, t_n)) . \]
\item A structure of ST $\k$-ring, with the topology from Definition 
\ref{dfn:41}, starting from the discrete topology on $\k'$.
\end{enumerate}
\end{dfn}

For $n = 0$ we have $K = \k'$, a finite extension of $\k$ with the discrete 
topology. The next definition is from \cite[Section 2.1]{Ye1}.

\begin{dfn} \label{dfn:12}
Let $\k$ be a perfect field.
An  {\em $n$-dimensional topological local field over $\k$}, for $n \geq 0$, 
is a field $K$, together with:
\begin{itemize}
\item[(a)] A structure $\{ \OO_i(K) \}_{i = 1}^n$ of $n$-dimensional local 
field on $K$.

\item[(b)] A ring homomorphism $\k \to \OO(K)$, such that 
$\k \to \kk_n(K)$ is finite. 

\item[(c)] A topology on $K$, making it a semi-topological $\k$-ring. 
\end{itemize}
The condition is this:
\begin{itemize}
\item[(P)] There is a bijection 
\[ f :  \k'((t_1, \ldots, t_n))   \iso   K  \]
from the standard $n$-dimensional topological local field 
with last residue field $\k' := \kk_n(K)$.
The bijection $f$ must have these two properties:
\begin{enumerate}
\rmitem{i} $f$ is an isomorphism in $\cat{LF}^n / \k$ (i.e.\ 
it respects the valuations).
\rmitem{ii} $f$ is  an isomorphism in 
$\cat{STRing}_{\mrm{c}} \k$ (i.e.\ it respects the topologies).
\end{enumerate}
Such an isomorphism $f$ is called a {\em parametrization} of $K$. 
\end{itemize} 
\end{dfn}

The parametrization $f$ is not part of the structure of $K$; it is required to 
exist, but (as we shall see) there are many distinct parametrizations. 
We use the abbreviation ``TLF'' for ``topological local field''. 

\begin{dfn}
Let $K$ and $L$ be $n$-dimensional TLFs over $\k$. A {\em morphism of TLFs} 
$f : K \to L$ is a homomorphism of $\k$-rings satisfying these two  
conditions:
\begin{enumerate}
\rmitem{i} $f$ is a morphism of $n$-dimensional local fields (i.e.\ it respects 
the valuations; see Definition \ref{dfn:15}). 
\rmitem{ii} $f$ is a homomorphism of ST $\k$-rings (i.e.\ it is continuous). 
\end{enumerate}
The category of $n$-dimensional TLFs over $\k$ is denoted by 
$\cat{TLF}^n / \k$. 
\end{dfn}

There are forgetful functors 
$\cat{TLF}^n  \k \to \cat{LF}^n  \k$ and 
$\cat{TLF}^n  \k \to \cat{STRing}_{\mrm{c}} \k$.

\begin{rem} \label{rem:70}
The conditions of Definition \ref{dfn:11} and \ref{dfn:12} are more 
restrictive than those of \cite[Definition 2.1.10]{Ye1}, in this respect: 
here we require that the last residue field $\k' := \kk_n(K)$ is finite over 
the 
base field $\k$, whereas in loc.\ cit.\ we only required that 
$\Om^1_{\k' / \k}$ should be a finite $\k'$-module (which allows $\k'$ to be a 
finitely generated extension field of $\k$ with transcendence degree $> 0$). 

If the TLF $K$ arises as a local factor of a Beilinson completion 
$\kk(x_0)_{\xi}$, as in Theorem \ref{thm:50}, then the last residue field 
$\kk_n(K)$ is finite over $\k$. So this fits into Definition \ref{dfn:12}.
\end{rem}

\begin{rem} \label{rem:30}
In \cite[Section 2.1]{Ye1} we also allow the much more general possibility of
a morphism of TLFs $f : K \to L$ where $\opn{dim} (K) < \opn{dim} (L)$. 
For instance, the inclusions 
$\k  \to \k((t_2)) \to \k((t_1, t_2))$ are morphisms. 
In this way we get a category $\cat{TLF}  \k$, that contains each 
$\cat{TLF}^n  \k$ as a full subcategory. 
\end{rem}

\begin{rem} \label{rem:31}
The papers on higher local fields from the Parshin school (prior to 1992) did 
not have a correct treatment of the topology on higher local fields. Some papers 
(e.g.\ \cite{Pa1, Be1, Pa3}) ignored it. Others -- most notably \cite{Lo} -- 
erroneously claimed that the topology of a local field is intrinsic, namely 
that it is determined by the valuations. This is correct in dimension $1$; but 
it is {\em false} when the dimension is $\geq 2$ and 
$\opn{char}(\k) = 0$. We gave a counterexample in \cite[Example 2.1.22]{Ye1}, 
that we reproduce in an expanded form as Example \ref{exa:30} below. 

It is a deep fact, also proved in \cite{Ye1}, that in characteristic 
$p > 0$ the topology is determined by the valuation, so that the forgetful 
functor 
$\cat{TLF}^n  \k \to \cat{LF}^n  \k$
is an equivalence.  
The proof relies on the structure of the ring of differential operators
$\DD_{K / \k}$ in characteristic $p > 0$
(see \cite[Theorem 2.1.14 and Proposition 2.1.21]{Ye1}).
\end{rem}

\begin{exa} \label{exa:30}
This is a slightly expanded version of \cite[Example 2.1.22]{Ye1}.
Let $\k$ be a field of characteristic $0$, and let 
$K := \k((t_1, t_2))$, the standard TLF of dimension $2$.
We choose a collection $\{ b_i \}_{i \in I}$ of elements in 
$\kk_1(K) = \k((t_2))$
that is a transcendence basis over the subfield $\k(t_2)$. 
For any $i \in I$ we choose some element 
$c_i \in \OO_1(K)$. As explained in the proof of Theorem \ref{thm:40}, there is 
a unique lifting 
\[ \si : \k((t_2)) \to  \OO_1(K) = \k((t_2))[[t_1]] \]
in $\cat{Ring}_{\mrm{c}} \k$ 
such that $\si(t_2) = t_2$, and 
$\si(b_i) = b_i + t_1 c_i$ for all $i \in I$.  
Next we extend $\si$ to a $\k$-ring automorphism 
$f : \OO_1(K) \to \OO_1(K)$ by setting $f(t_1) := t_1$. 
By localization this extends to a $\k$-ring automorphism $f : K \to K$.

It easy to check that $f$ is an automorphism of $K$ in 
the category $\cat{LF}^2  \k$ of local fields.  
However, since $f$ is the identity on the subfield 
$\k(t_1, t_2) \subset K$, and this subfield is a dense subset of $K$, it 
follows that $f$ is continuous 
iff it is the identity automorphism of $K$, 
iff $c_i = 0$ for all $i$. Thus if we choose at 
least one $c_i \neq 0$, $f$ is not a morphism in $\cat{TLF}^2  \k$.
\end{exa}

Let $K$ be a TLF of dimension $n \geq 1$ over $\k$.
The inclusion $\OO_1(K) \inj K$ gives $\OO_1(K)$ an induced structure of ST 
$\k$-ring (it is the subspace topology). Then the surjection 
$\OO_1(K) \surj \kk_1(K)$ gives $\kk_1(K)$ an induced structure of ST 
$\k$-ring (it is the quotient topology). And so on all the way to $\kk_n(K)$.
In other words, the topologies are such that each 
$\OO_i(K) \inj \kk_{i - 1}(K)$ is a strict monomorphism in 
$\cat{STRing}_{\mrm{c}} \k$, and each 
$\OO_i(K) \surj \kk_{i}(K)$ is a strict epimorphism.
 
If we choose a parametrization 
$K \cong \k'((t_1, \ldots, t_n))$, 
then the induced ring isomorphisms 
\[ \k'((t_{i + 1}, \ldots, t_n))[[t_i]] \cong \OO_i(K) \]
and
\[ \k'((t_{i + 1}, \ldots, t_n)) \cong \kk_i(K) \]
are also isomorphisms of ST $\k$-rings. 
This follows from \cite[Proposition 1.3.5]{Ye1}.
In particular, each $\kk_i(K)$ is an $(n - i)$-dimensional TLF over $\k$. 

Recall the notions of precise lifting and precise artinian local ring from 
Definition \ref{dfn:23}.

\begin{lem} \label{lem:25}
Let $K$ be a TLF of dimension $n \geq 1$ over $\k$, let $l \geq 0$.
Then the ST ring $A_l := \OO_1(K) / \m_1(K)^{l + 1}$,
with the quotient topology from $\OO_1(K)$, is a precise artinian local ring in 
$\cat{STRing}_{\mrm{c}} \k$. 
\end{lem}

\begin{proof}
Choose a parametrization $K \cong \k'((t_1, \ldots, t_n))$,
and let $\bar{K} := \k'((t_2, \ldots, t_n))$. Then 
$\bar{K} \cong \kk_1(K)$ and 
$\bar{K}[[t_1]] \cong \OO_1(K)$ as ST $\k$-rings; and the inclusion 
$\bar{K} \to \bar{K}[[t_1]]$ represents a lifting 
$\si_1 : \kk_1(K) \to \OO_1(K)$.
As ST $\bar{K}$-modules we have 
$\OO_1(K) \cong \prod_{i = 0}^{\infty} \bar{K}$
and 
$A_l \cong \prod_{i = 0}^{l} \bar{K}$.
This shows that the quotient topology on $A_l$ coincides with the fine 
$\bar{K}$-module topology on it. So $\si_1$ is a precise lifting. 
\end{proof}

\begin{lem}
Let $K \in \cat{TLF}^n  \k$, with last residue field $\k' := \kk_n(K)$.
There is a unique lifting $\si : \k' \to \OO(K)$ in 
$\cat{STRing}_{\mrm{c}} \k$ of the canonical surjection 
$\OO(K) \surj \k'$.
\end{lem}

\begin{proof}
Since $\k'$ is discrete, we do not have to worry about continuity. 
We use induction on $n$. Let 
$\bar{\si} : \k' \to \OO(\kk_1(K)) \subset \kk_1(K)$ 
be the unique lifting for this $(n - 1)$-dimensional TLF.
Consider the canonical surjection $\pi : \OO_1(K) \to \kk_1(K)$. 
By Theorem \ref{thm:40} there is a unique $\k$-ring homomorphism 
$\si : \k' \to \OO_1(K)$ such that 
$\pi \circ \si = \bar{\si}$. 
It is trivial to see that $\si(\k')$ is inside $\OO(K)$.
\end{proof}

The construction / classification of parametrizations of a TLF (condition (P) 
in Definition \ref{dfn:12}) is made clear by the next theorem (which is a 
special case of \cite[Corollary 2.1.19]{Ye1}). 

\begin{thm}[\cite{Ye1}] \label{thm:25}
Let $K$ be an $n$-dimensional TLF over $\k$, let $(a_1, \ldots, a_n)$ be a 
system of uniformizers in $K$, let $\k' := \kk_n(K)$, and let 
$\si : \k' \to \OO(K)$ be the unique lifting over $\k$. Then $\si$ 
extends uniquely to an isomorphism of TLFs 
\[ f : \k'((t_1, \ldots, t_n)) \to K \]
such that $f(t_i) = a_i$. 
\end{thm}

\begin{dfn}
Let $K$ be an $n$-dimensional TLF over $\k$. By a {\em system of liftings} for 
$K$ we mean a sequence 
$\bsym{\si} = (\si_1, \ldots, \si_n)$, where for each $i$
\[ \si_i : \kk_i(K) \to \OO_i(K) \]
is a homomorphism of ST $\k$-rings that lifts the canonical surjection
$\OO_i(K) \surj \kk_i(K)$. 
\end{dfn}

The important thing to remember is that each lifting $\si_i : \kk_i(K) \to 
\OO_i(K)$ is continuous. 
When $n = 0$ the only system of liftings is the empty system 
$\bsym{\si} = ()$. 

\begin{exa}
Take a standard TLF $K := \k'((t_1, \ldots, t_n))$. It comes equipped with a 
standard system of liftings
\[ \si_i : \kk_i(K) \to \OO_i(K) , \]
namely the inclusions 
\[ \si_i : \k'((t_{i + 1}, \ldots, t_n)) \to 
\k'((t_{i + 1}, \ldots, t_n))[[t_i]] . \]
\end{exa}

\begin{prop} \label{prop:150}
Any $n$-dimensional TLF $K$ over $\k$ admits a system of liftings. 
\end{prop}

\begin{proof}
Take a parametrization 
$f : \k'((t_1, \ldots, t_n)) \to K$.
The standard system of liftings of $\k'((t_1, \ldots, t_n))$
induces a system of liftings on $K$. 
\end{proof}

%%\cleardoublepage
\section{Lattices and BT Operators}
\label{sec:lat-BT}

As before, $\k$ is a perfect base field. 

\begin{dfn}
Let $K$ be an $n$-dimensional TLF over $\k$, and let $M$ be a finite 
$K$-module. An {\em $\OO_1(K)$-lattice} in $M$ is a finite $\OO_1(K)$-submodule 
$L$ of $M$ such that $M = K \cd L$. 
We denote by $\opn{Lat}(M)$ the set of $\OO_1(K)$-lattices in $M$. 
\end{dfn}

Let $L$ be an $\OO_1(K)$-lattice in $M$. Recall that $\OO_1(K)$
is a DVR. This implies that $L$ is a free $\OO_1(K)$-module, of rank equal to 
that of $M$. 

\begin{exa}
Consider a TLF $K$, and take $M := K^r$. Choose a uniformizer 
$a \in \OO_1(K)$. For any $i \in \Z$ there is a 
lattice $L_i := a^i \cd \OO_1(K)^r \subset K^r$. Let us call these standard 
lattices. They do not depend on the choice of uniformizer. 

When $r = 1$, all the $\OO_1(K)$-lattices in $M$ are standard.
When $r > 1$, $M$ has many more lattices. However any 
$\OO_1(K)$-lattice $L$ in $M$ can be sandwiched between two standard 
lattices: $L_i \subset L \subset L_{-j}$ for $i, j \gg 0$. 
\end{exa}
 
Suppose $M$ is a finite $K$-module, and $L, L' \in \opn{Lat}(M)$, 
with $L \subset L'$. Then the quotient $L' / L$ is a finite length 
$\OO_1(K)$-module. If we are given a lifting 
$\si_1 : \kk_1(K) \to \OO_1(K)$, then $L' / L$ becomes a finite module over the 
TLF $\kk_1(K)$, that we denote by $\opn{rest}_{\si_1}(L' / L)$; cf.\ Definition 
\ref{dfn:65}.

\begin{lem} \label{lem:50}
Let $M$ be a finite $K$-module, let $L$ be an $\OO_1(K)$-lattice in $M$, and 
let $a \in \OO_1(K)$ be a uniformizer. 
Give $M$ the fine $K$-module topology.
For every $i \in \Z$ give the lattice $L_i := a^{i} \cd L$ the fine 
$\OO_1(K)$-module topology. For every $i \in \N$ give the quotient
$L / L_{i}$ the fine $\OO_1(K)$-module topology. 
\begin{enumerate}
\item The topology on $M$ equals the fine $\OO_1(K)$-module topology on it. 
\item The inclusions $L_i \to M$, for $i \in \Z$, are strict 
monomorphisms in $\cat{STMod} \k$.
\item Consider the direct system $\{ L_{-j} \}_{j \in \N}$
in $\cat{STMod} \k$. 
Give $\lim_{j \to} L_{-j}$ the direct limit topology. The the canonical 
bijection $\lim_{j \to} L_{-j} \to M$ is an isomorphism in $\cat{STMod} \k$.
\item The canonical surjections $L \to L / L_i$, for $i \in \N$, are 
strict epimorphisms in $\cat{STMod} \k$.
\item Let $\si_1 : \kk_1(K) \to \OO_1(K)$ be a lifting in 
$\cat{STRing}_{\mrm{c}} \k$ of the canonical surjection. Then for every 
$i \in \N$ the topology on $L/ L_{i}$ equals the the fine 
$(\kk_1(K), \si_1)$-module topology on it. 
\item Consider the inverse system $\{ L / L_{i} \}_{i \in \N}$
in $\cat{STMod} \k$. 
Give $\lim_{\leftarrow i} (L / L_{i})$ the inverse limit topology. Then the 
canonical bijection $L \to \lim_{\leftarrow i} (L / L_{i})$ is an  
isomorphism in $\cat{STMod} \k$.
\end{enumerate}
\end{lem}

\begin{proof}
All these assertions become clear after we choose an $\OO_1(K)$-linear
isomorphism $L \cong \OO_1(K)^r$ and a ST $\k$-ring isomorphism
$\OO_1(K) \cong \kk_1(K)[[t]]$. 
Cf.\ \cite[Proposition 1.3.5]{Ye1}.
\end{proof}

Let $K$ be a TLF of dimension $n \geq 1$ over $\k$.
If $\bsym{\si} = (\si_1, \ldots, \si_n)$ is a system of liftings for $K$,
then we write
$\d_1(\bsym{\si}) := (\si_2, \ldots, \si_n)$. 
This is a system of liftings for the TLF $\kk_1(K)$.

\begin{dfn} \label{dfn:50}
Let $K$ be an $n$-dimensional TLF over $\k$, and let $(M_1, M_2)$ be a pair of 
finite $K$-modules.
\begin{enumerate}
\item By a pair of $\OO_{1}(K)$-lattices in $(M_1, M_2)$ 
we mean a pair $(L_1, L_2)$, where $L_i \in \opn{Lat}(M_i)$.
The set of such pairs is denoted by $\opn{Lat}(M_1, M_2)$.
\item Let $\phi : M_1 \to M_2$ be a $\k$-linear homomorphism, and let 
$(L_1, L_2), (L'_1, L'_2) \in \opn{Lat}(M_1, M_2)$.
We say that $(L'_1, L'_2)$ is a {\em $\phi$-refinement} of $(L_1, L_2)$ if 
$L'_1 \subset L_1$, $L_2 \subset L'_2$, 
$\phi(L'_1) \subset L_{2}$ and $\phi(L_{1}) \subset L'_{2}$.
In this case we write 
\[ (L'_1, L'_2) \prec_{\phi} (L_1, L_2) , \]
and refer to it as a {\em $\phi$-refinement} in $\opn{Lat}(M_1, M_2)$.
\end{enumerate}
\end{dfn}

The relation $\prec_{\phi}$ is a partial ordering on the set 
$\opn{Lat}(M_1, M_2)$. If $(L'_1, L'_2) \prec_{\phi} (L_1, L_2)$, then 
there is an induced $\k$-linear homomorphism 
$\bar{\phi} : L_1 / L'_1 \to L'_2 / L_2$.

The next two definitions are variations of the original definitions by 
Beilinson \cite{Be1}, that are themselves generalizations to $n \geq 2$ of the 
definitions of Tate \cite{Ta}. We saw similar definitions in the more recent 
papers \cite{Os1, Os2} and \cite{Br1, Br2}. The notation we use is close to 
that of Tate. 

\begin{dfn} \label{dfn:5}
Let $K$ be an $n$-dimensional TLF over $\k$, let
$\bsym{\si} = (\si_1, \ldots, \si_n)$ be a system of liftings for 
$K$, and let $(M_1, M_2)$ be a pair of finite $K$-modules.
We define the subset
\[ \mrm{E}^{K}_{\bsym{\si}}(M_{1}, M_{2}) \subset 
\opn{Hom}_{\k}^{}(M_{1}, M_{2})  \]
as follows. 
\begin{enumerate}
\item If $n = 0$ then any $\k$-linear homomorphism $\phi : M_{1} \to M_{2}$
belongs to \lb $\mrm{E}^{K}_{\bsym{\si}}(M_{1}, M_{2})$.
\item If $n \geq 1$, then a $\k$-linear homomorphism  
$\phi : M_{1} \to M_{2}$ belongs to \lb
 $\mrm{E}^{K}_{\bsym{\si}}(M_{1}, M_{2})$
if it satisfies these two conditions:
\begin{enumerate}
\rmitem{i} Every $(L_1, L_2) \in \opn{Lat}(M_1, M_2)$
has some $\phi$-refinement $(L'_1, L'_2)$.
\rmitem{ii} For every $\phi$-refinement 
$(L'_1, L'_2) \prec_{\phi} (L_1, L_2)$ in $\opn{Lat}(M_1, M_2)$
the induced homomorphism 
\[ \bar{\phi} : L_1 / L'_1 \to L'_2 / L_2 \]
belongs to 
\[ \mrm{E}^{\kk_1(K)}_{\d_1(\bsym{\si})}
\bigl( \opn{rest}_{\si_1}(L_1 / L'_1),
\opn{rest}_{\si_1}(L'_2 / L_2) \bigr) . \] 
\end{enumerate}
\end{enumerate}

A homomorphism $\phi : M_1 \to M_2$ that belongs to 
$\mrm{E}^{K}_{\bsym{\si}}(M_{1}, M_{2})$
is called a {\em local Beilinson-Tate operator} relative to $\bsym{\si}$, or a 
BT operator for short. 
\end{dfn}

Let $K$ be a TLF over $\k$ of dimension $\geq 1$.
We denote by $\OO_1(K)\|_{\m_1(K)}$ the ST $\k$-ring which is the ring
$\OO_1(K)$ with its $\m_1(K)$-adic topology.
Given an $\OO_1(K)$-module $M$, the fine $\OO_1(K)\|_{\m_1(K)}$-module 
topology on $M$ is called the fine $\m_1(K)$-adic topology. 
Now suppose $\phi : M_1 \to M_2$ is a $\k$-linear homomorphism.
We say that $\phi$ is {\em $\m_1(K)$-adically continuous} 
if it is continuous for the fine $\m_1(K)$-adic topologies on $M_1$ and $M_2$.

\begin{exa} \label{exa:63}
If $n = 1$ then the usual topology on $\OO_1(K)$ equals the 
$\m_1(K)$-adic topology. Thus $K$ has the fine $\m_1(K)$-adic topology. 
If $n > 1$ then the fine $\m_1(K)$-adic topology is finer than, and not 
equal to, the usual topology on $\OO_1(K)$ and $K$. 
\end{exa}

\begin{lem} \label{lem:51}
Consider the situation of Definition \tup{\ref{dfn:5}}.
The homomorphism $\phi : M_1 \to M_2$ satisfies condition \tup{(2.i)} iff it is
$\m_1(K)$-adically continuous. 
\end{lem}

\begin{proof}
Let $\bar{K}$ be the field $\kk_1(K)$, but with the discrete topology. The 
lifting $\si_1$ induces an isomorphism of ST rings 
$\bar{K}[[t]] \iso \OO_1(K)\|_{\m_1(K)}$. 
Thus the field $K$, with the fine $\m_1(K)$-adic topology, is isomorphic to 
$\bar{K}((t))$ as ST $\k$-rings. But we know that 
\[ \bar{K}((t)) \cong 
\Bigl( \prod\nolimits_{i \geq 0} \bar{K} \cd t^i \Bigr) 
\oplus 
\Bigl( \bigoplus\nolimits_{i < 0} \bar{K} \cd t^i \Bigr) \]
as ST $\k$-modules, where $\bar{K} \cd t^i \cong \bar{K}$ is discrete; cf.\ 
proof of \cite[Proposition 1.3.5]{Ye1}. It is now an exercise in quantifiers to 
compare $t$-adic continuity to condition \tup{(2.i)}. Cf.\ 
\cite[Remark 1 in Section 1.1]{Br2} where this is also mentioned. 
\end{proof}

\begin{lem} \label{lem:44}
In the situation of Definition \tup{\ref{dfn:5}}, give $M_1$ and $M_2$ the fine 
$K$-module topologies. Then every 
$\phi \in \mrm{E}^{K}_{\bsym{\si}}(M_{1}, M_{2})$
is continuous.
\end{lem}

\begin{proof}
The proof is by induction on $n$. For $n = 0$ there is nothing to prove, since 
these are discrete modules. So assume $n \geq 1$. (Actually for $n = 1$ this 
was proved in Lemma \ref{lem:51}.) 
In view of Lemma \ref{lem:50}, it suffices to prove that for  
every $(L'_1, L'_2) \prec_{\phi} (L_1, L_2)$ in $\opn{Lat}(M_1, M_2)$,
the induced homomorphism 
$\bar{\phi} : L_1 / L'_1 \to L'_2 / L_2$ is continuous. 
But $\bar{\phi}$ is a BT operator in dimension $n - 1$, so by induction it is 
continuous. 
\end{proof}

\begin{lem} \label{lem:52}
Let $K$ be an $n$-dimensional TLF over $\k$, and let 
$\bsym{\si}$ be a system of liftings for $K$. For $l = 1, 2, 3, 4$ let 
$M_l$ be a finite $K$-module, and for  $l = 1, 2, 3$ let 
$\phi_l : M_l \to M_{l + 1}$ be a $\k$-linear homomorphism. 
\begin{enumerate}
\item If $\phi_1$ is $K$-linear then it is a BT operator.
\item If $\phi_1$ and $\phi_2$ are BT operators, then $\phi_2 \circ \phi_1$ is 
a BT operator.
\item Assume that $\phi_1$ is surjective and $K$-linear, $\phi_3$ is injective 
and $K$-linear, and $\phi_3 \circ \phi_2 \circ \phi_1$ is a BT 
operator. Then $\phi_2$ is a BT operator. 
\end{enumerate}
\end{lem}

Here is a diagram depicting the situation:
\[ M_1 \xar{\phi_1} M_2 \xar{\phi_2} M_3 \xar{\phi_3} M_4 . \]

\begin{proof}
We prove all three assertions by induction on $n$ and on their 
sequential order. For $n = 0$ all assertions are trivial, 
so let us assume that $n \geq 1$. The conditions mentioned below are those in 
Definition \ref{dfn:5}.

\medskip \noindent
(1) For this we assume that assertion (1) is true in dimension $n - 1$. 
Condition (2.i), namely the $\m_1(K)$-adic continuity of $\phi_1$, is clear. 
Consider any $\phi$-refinement 
$(L'_1, L'_2) \prec_{\phi_1} (L_1, L_2)$ in $\opn{Lat}(M_1, M_2)$.
Since the induced homomorphism 
$\bar{\phi} : L_1 / L'_1 \to L'_2 / L_2$ is $\OO_1(K)$-linear, it is also 
$(\kk_1(K), \si_1)$-linear. By induction on $n$, $\bar{\phi}$ is a BT operator.
So condition (2.ii) holds.

\medskip \noindent
(2) Here we assume that assertions (2) and (3) are true in dimension $n - 1$.
Write $\psi := \phi_2 \circ \phi_1$. The $\m_1(K)$-adic continuity of 
$\psi$, i.e.\ condition (2.i), is clear. Consider any $\psi$-refinement
$(L'_1, L'_3) \prec_{\psi} (L_1, L_3)$ in $\opn{Lat}(M_1, M_3)$.
To satisfy condition (2.ii) we have to prove that 
$\bar{\psi} : L_1 / L'_1 \to L'_3 / L_3$ is a BT operator in dimension $n - 1$. 
Let $L_2^{\Diamond} \in \opn{Lat}(M_2)$ be a lattice that contains 
$\phi_1(L_1)$, and let $L_3^{\Diamond} \in \opn{Lat}(M_3)$  be a lattice that 
contains both $L'_3$ and $\phi_2(L_2^{\Diamond})$. 
Let $L_2^{\heartsuit} \in \opn{Lat}(M_2)$ be a lattice that
is contained in $L_2^{\Diamond}$. 
Let $L_1^{\heartsuit} \in \opn{Lat}(M_1)$ be such that 
$L_1^{\heartsuit} \subset L_1'$ and 
$\phi_1(L_1^{\heartsuit}) \subset L_2^{\heartsuit}$. 
All these choices are possible because condition (2.i) is satisfied by 
$\phi_1$ and $\phi_2$. 
Consider the commutative diagram 
\[ \UseTips \xymatrix @C=6ex @R=3ex { 
\displaystyle \frac{L_1}{L_1^{\heartsuit}}
\ar@{->>}[r]^{\al}
\ar[dr]_{\bar{\phi}_1}
&
\displaystyle \frac{L_1}{L'_1}
\ar[r]^{\bar{\psi}}
&
\displaystyle \frac{L'_3}{L_3}
\ar@{>->}[r]^{\be}
&
\displaystyle \frac{L_3^{\Diamond}}{L_3}
\\
&
\displaystyle \frac{L_2^{\Diamond}}{L_2^{\heartsuit}}
\ar[urr]_{\bar{\phi}_2}
} \]
in $\cat{Mod} \k$.
Since $\phi_1$ and $\phi_2$ are BT operators, condition (2.ii) says that 
$\bar{\phi}_1$ and $\bar{\phi}_2$ are BT operators (in dimension $n - 1$). 
By part (2) the composition $\bar{\phi}_2 \circ \bar{\phi}_1$ is a BT operator. 
The homomorphisms $\al$ and $\be$ are $\kk_1(K)$-linear. Therefore by part (3) 
the homomorphism $\bar{\psi}$ is a BT operator. 

\medskip \noindent
(3) For this we assume that assertions (1) and (2) are true in dimension $n$. 
Let $\psi := \phi_3 \circ \phi_2 \circ \phi_1$. 
Choose $K$-linear homomorphisms 
$\psi_1 : M_2 \to M_1$ and $\psi_3 : M_4 \to M_3$ that split $\phi_1$ and 
$\phi_3$ respectively. Then 
$\phi_2 = \psi_3 \circ \psi \circ \psi_1$. By assertions 
(1) and (2) we see that $\phi_2$ is a BT operator. 
\end{proof}

\begin{lem} \label{lem:60}
In the situation of Definition \tup{\ref{dfn:5}}, 
The set $\mrm{E}^{K}_{\bsym{\si}}(M_{1}, M_2)$ is a $\k$-submodule of 
$\opn{Hom}_{\k}^{}(M_{1}, M_{2})$.
\end{lem}

\begin{proof}
The proof is by induction on $n$, and we can assume that $n \geq 1$. 
Take any $\phi_1, \phi_2 \in  \mrm{E}^{K}_{\bsym{\si}}(M_{1}, M_2)$
and any $a \in \k$. Let $\psi := a \cd \phi_1 + \phi_2$; we have to show that 
$\psi \in  \mrm{E}^{K}_{\bsym{\si}}(M_{1}, M_2)$.
Since condition (2.i) of  Definition \ref{dfn:5} is about $\m_1(K)$-adic 
continuity (by Lemma \ref{lem:51}), we see that $\psi$ satisfies 
it. 

We need to check condition (2.ii) of that definition. So let $(L'_1, L'_2)$ be 
a $\psi$-refinement of $(L_1, L_2)$. By $\m_1(K)$-adic continuity there are
lattices 
$L_1^{\Diamond} \subset L'_1$
and $L'_2 \subset L_2^{\Diamond}$
such that $\phi_i(L_1^{\Diamond}) \subset L_2$ and 
$\phi_i(L_1) \subset L_2^{\Diamond}$. 
Consider the commutative diagram 
\[ \UseTips \xymatrix @C=6ex @R=6ex {
\displaystyle \frac{L_1}{L_1^{\Diamond}}
\ar@{->>}[r]^{\al}
\ar@(d,d)[rrr]_{a \cd \bar{\phi}_1 + \bar{\phi}_2}
&
\displaystyle \frac{L_1}{L'_1}
\ar[r]^{\bar{\psi}}
&
\displaystyle \frac{L'_2}{L_2}
\ar@{>->}[r]^{\be}
&
\displaystyle \frac{L_2^{\Diamond}}{L_2}
} \]
in $\cat{Mod} \k$.
The induction hypothesis tells us that 
$a \cd \bar{\phi}_1 + \bar{\phi}_2$ is a BT operator. The homomorphisms $\al$ 
and $\be$ are $\kk_1(K)$-linear. Therefore according to Lemma \ref{lem:52}(3) 
the homomorphism $\bar{\psi}$ is a BT operator. 
\end{proof}

\begin{lem} \label{lem:67}
Let $K$ be an $n$-dimensional TLF over $\k$, let 
$\bsym{\si}$ be a system of liftings for $K$, and let $(M_1, M_2)$ be a pair of 
finite $K$-modules. Then
\[ \opn{Diff}_{K / \k}^{\mrm{cont}}(M_1, M_2) \subset
\opn{E}^{K}_{\bsym{\si}}(M_1, M_2) . \]
\end{lem}

\begin{proof}
We use induction on $n$. For $n = 0$ there is nothing to prove, so let's 
assume that $n \geq 1$. (Actually, for $n = 1$ there is nothing to prove 
either; cf.\ Example \ref{exa:60}.) 

Let $\phi : M_1 \to  M_2$ be a continuous differential operator. 
Choose $K$-linear isomorphisms 
$M_l \cong K^{r_l}$ for $l = 1, 2$; so we may view $\phi$ as a matrix 
$[\phi_{i, j}] \in \opn{Mat}_{r_2 \times r_1}(\DD^{\mrm{cont}}_{K / \k})$.
According to Lemma \ref{lem:52}(1, 2), it suffices to prove that each 
$\phi_{i, j}$ is a BT operator. Therefore we can assume that 
$M_1 = M_2 = K$ and $\phi \in \DD^{\mrm{cont}}_{K / \k}$.

Choose a uniformizer $a \in \OO_1(K)$.  
If $\opn{char}(\k) = 0$ then by formula (\ref{eqn:60}) there is an integer 
$d$, depending on the coefficients of the operator $\phi$ in that expansion, 
such that $\phi(a^i \cd \OO_1(K)) \subset a^{i - d} \cd \OO_1(K)$
for all $i$. Hence $\phi$ is $\m_1(K)$-adically continuous. 

If $\opn{char}(\k) = p > 0$, then by \cite[Theorem 1.4.9]{Ye1} the operator 
$\phi$ is linear over the subfield $K' := \k \cd K^{p^d} \subset K$
for a sufficiently large natural number $d$. The field $K'$ is also an 
$n$-dimensional TLF, $K' \to K$ is a morphism of TLFs, and 
$\OO_1(K') \to \OO_1(K)$ is a finite homomorphism.
So the $\m_1(K)$-adic topology on $\OO_1(K)$ coincides with its 
$\m_1(K')$-adic topology. Since $\phi$ is  $\OO_1(K')$-linear,
it follows that $\phi$ is $\m_1(K')$-adically continuous.
Using Lemma \ref{lem:51} we see that in both cases ($\opn{char}(\k) = 0$ and 
$\opn{char}(\k) > 0$) condition 
(2.i) of Definition \ref{dfn:5} holds. 

Now take a $\phi$-refinement 
$(L'_1, L'_2) \prec_{\phi} (L_1, L_2)$ in $\opn{Lat}(M_1, M_2)$.
Write 
$\bar{M}_1 := \lb \opn{rest}_{\si_1}(L_1 / L'_1)$ and 
$\bar{M}_2 := \opn{rest}_{\si_1}(L'_2 / L_2)$.
We must prove that 
$\bar{\phi} : \bar{M}_1 \to \bar{M}_2$ is a BT operator between these 
$\kk_1(K)$-modules.
We know that $\bar{\phi}$ is a differential operator over $\OO_1(K)$, and 
therefore it is also a differential operator over $\kk_1(K)$.
Choose some $\kk_1(K)$-linear isomorphisms 
$\psi_l : \kk_1(K)^{r_l} \iso \bar{M}_l$.
Then 
\[ \psi := \psi_2^{-1} \circ \bar{\phi} \circ \psi_1 : \kk_1(K)^{r_1} \to 
\kk_1(K)^{r_2} \]
is a differential operator over $\kk_1(K)$. By the induction hypothesis, 
$\psi$ is a BT operator. Finally by Lemma \ref{lem:52}(1, 2) the homomorphism 
$\bar{\phi} = \psi_2 \circ \psi \circ \psi_1^{-1}$ 
is a BT operator.
\end{proof}

\begin{exa} \label{exa:60}
If $n = 0$ then by definition
\[ \mrm{E}^{K}_{\bsym{\si}}(M_{1}, M_{2}) =
\opn{Hom}_{\k}(M_{1}, M_{2}) . \]
This is a finite $\k$-module. 

If $n = 1$ then condition (ii.b) of Definition \ref{dfn:5} is trivially 
satisfied. Lemma \ref{lem:51} and Example \ref{exa:60} show that 
\[ \mrm{E}^{K}_{\bsym{\si}}(M_{1}, M_{2}) =
\opn{Hom}_{\k}^{\mrm{cont}}(M_{1}, M_{2}) , \]
the module of continuous $\k$-linear homomorphisms. This was already noticed in
\cite{Os1, Os2} and \cite[Section 1.1]{Br2}.

The equalities above indicate that the choice of $\bsym{\si}$ is irrelevant.
However in dimension $\leq 1$ there is only one lifting, so in fact there is no 
news here. Later, in Theorem \ref{thm:65}, we will prove that in any dimension 
the system of liftings $\bsym{\si}$ is not relevant.  
\end{exa}

\begin{exa}
For  $n \geq 2$ the inclusion 
\[  \mrm{E}^{K}_{\bsym{\si}}(M_{1}, M_{2}) \subset
\opn{Hom}_{\k}^{\mrm{cont}}(M_{1}, M_{2}) \]
is usually proper (i.e.\ it is not an equality). Here is a calculation 
demonstrating this. Let $K := \k((t_1, t_2))$, the standard TLF with its 
standard 
system of liftings $\bsym{\si}$.
Take $M_1 = M_2 := K$.
Any $a \in K$ is a series 
$a = \sum_{i \in \Z} a_i(t_2) \cd t_1^i$, where $a_i(t_2) \in \k((t_2))$, and 
$a_i(t_2) = 0$ for $i \ll 0$. We let 
$\psi \in \opn{End}_{\k}(K)$ be 
\[ \psi \Bigl( \sum\nolimits_{i \in \Z} a_i(t_2) \cd t_1^i \Bigr) := a_0(t_1) 
. \] 
To see that this is continuous we use the  continuous decomposition
\[ K = \k((t_2))((t_1)) \cong \k((t_2))[[t_1]] \oplus 
\Bigl( \bigoplus\nolimits_{i < 0} \k((t_2)) \cd t_1^i \Bigr) . \]
This gives a continuous function $\psi_1 : K \to \k((t_2))$,
sending 
$\sum_{i \in \Z} a_i(t_2) \cd t_1^i \mapsto a_0(t_2)$.
Next there is an isomorphism $\psi_2 : \k((t_2)) \to \k((t_1))$,
$a_0(t_2) \mapsto a_0(t_1)$. Finally the inclusion 
$\psi_3 : \k((t_1)) \to \k((t_2))((t_1))$ is continuous.
The function $\psi$ is 
$\psi = \psi_3 \circ \psi_2 \circ \psi_1$, so it is continuous.

Take the standard lattices $L_i = t_1^i \cd \k((t_2))[[t_1]]$ in $K$. 
For every $j$ the element $a_j := t_2^j$ belongs to $L_0$, 
yet the element $\psi(a_j) = t_1^j$ does not belong to $L_{j + 1}$.
Thus $\psi(L_0)$ is not contained in any lattice, and requirement (i) is 
violated, so $\psi$ does not belong to 
$\mrm{E}^{K}_{\bsym{\si}}(K, K)$. 
\end{exa}

Recall that for an $n$-dimensional TLF $K$, with $n \geq 2$, and a system 
of liftings $\bsym{\si} = (\si_1, \ldots, \si_n)$, 
the truncation $\d_1(\bsym{\si}) = (\si_2, \ldots, \si_n)$
is a system of liftings for the first residue field $\kk_1(K)$.  

\begin{dfn} \label{dfn:73}
Let $K$ be a TLF over $\k$ of dimension $n \geq 1$, let 
$\bsym{\si} = (\si_1, \ldots, \si_n)$ be a 
system of liftings for $K$, and let $(M_1, M_2)$ be a pair of finite 
$K$-modules. For integers $i \in \{ 1, \dots, n \}$ and $j \in \{ 1, 2 \}$, we 
define the subset
\[ \mrm{E}^{K}_{\bsym{\si}}(M_{1}, M_{2})_{i, j} \subset
\mrm{E}^{K}_{\bsym{\si}}(M_{1}, M_{2}) \]
to be the set of BT operators $\phi : M_1 \to M_2$
that satisfy the conditions below. 
\begin{itemize}
\rmitem{i} The operator $\phi$ belongs to 
$\mrm{E}^{K}_{\bsym{\si}}(M_{1}, M_{2})_{1, 1}$ if there 
exists some $L_2 \in \opn{Lat}(M_2)$ such that 
$\phi(M_{1}) \subset L_{2}$. 
\rmitem{ii}  The operator $\phi$ belongs to 
$\mrm{E}^{K}_{\bsym{\si}}(M_{1}, M_{2})_{1, 2}$
if there exists some $L_1 \in \opn{Lat}(M_1)$ such that 
$\phi(L_{1}) = 0$. 
\rmitem{iii} This only refers to $n \geq 2$. 
For  $i \in \{ 2, \dots, n \}$ and 
$j \in \{ 1, 2 \}$, the operator $\phi$ belongs to 
$\mrm{E}^{K}_{\bsym{\si}}(M_{1}, M_{2})_{i, j}$
if for any $\phi$-refinement
$(L'_1, L'_2) \prec_{\phi} (L_1, L_2)$ in $\opn{Lat}(M_1, M_2)$
the induced homomorphism 
\[ \bar{\phi} : L_1 / L'_1 \to L'_2 / L_2 \]
belongs to 
\[ \mrm{E}^{\kk_1(K)}_{\d_1(\bsym{\si})}
\bigl( \opn{rest}_{\si_1}(L_1 / L'_1),
\opn{rest}_{\si_1}(L'_2 / L_2) \bigr)_{i - 1, j}  . \] 
\end{itemize}
\end{dfn}

\begin{dfn} \label{dfn:100}
Let $K$ be an $n$-dimensional TLF over $\k$, and let $\bsym{\si}$ be a system 
of liftings for $K$. We define 
\[ \mrm{E}_{\bsym{\si}}(K) := \mrm{E}^{K}_{\bsym{\si}}(K, K) . \]
If $n \geq 1$ we define 
\[ \mrm{E}_{\bsym{\si}}(K)_{i, j} := \mrm{E}^{K}_{\bsym{\si}}(K, K)_{i, j} . \]
\end{dfn}

\begin{lem} \label{lem:66}
Let $K$ be an $n$-dimensional TLF over $\k$, with $n \geq 1$, and let 
$\bsym{\si}$ be a system of liftings for $K$. For $l = 1, 2, 3, 4$ let 
$M_l$ be a finite $K$-module, and for $l = 1, 2, 3$ let 
$\phi_l \in \mrm{E}^{K}_{\bsym{\si}}(M_{l}, M_{l + 1})$.
Take any $j \in \{ 1, 2 \}$ and $i \in \{ 1, \ldots, n \}$.
\begin{enumerate}
\item The set $\mrm{E}^{K}_{\bsym{\si}}(M_{1}, M_2)_{i, j}$ is a $\k$-submodule 
of $\mrm{E}^{K}_{\bsym{\si}}(M_{1}, M_2)$. 
\item If 
$\phi_2 \in \mrm{E}^{K}_{\bsym{\si}}(M_{2}, M_{3})_{i, j}$,
then 
$\phi_3 \circ \phi_2 \circ \phi_1 \in 
\mrm{E}^{K}_{\bsym{\si}}(M_{1}, M_4)_{i, j}$.
\item  Assume that $\phi_1$ is surjective and $K$-linear, $\phi_3$ is injective 
and $K$-linear, and 
$\phi_3 \circ \phi_2 \circ \phi_1 
\in \mrm{E}^{K}_{\bsym{\si}}(M_{1}, M_4)_{i, j}$.
Then $\phi_2 \in \mrm{E}^{K}_{\bsym{\si}}(M_{2}, M_3)_{i, j}$.
\end{enumerate}
\end{lem}

\begin{proof}
We use induction on $n$ and on the sequential order of the assertions. 

\medskip \noindent
(1) For $i = 1$ this is clear. Now assume $i \geq 2$ (and hence also 
$n \geq 2$). For this we use the same strategy as in the proof of Lemma 
\ref{lem:60}. We are allowed to make use of assertion (3) in dimension 
$n - 1$. 

\medskip \noindent
(2) For $i = 1$ this is clear. Now assume $i \geq 2$ (and hence also 
$n \geq 2$). Here we use the same proof as of Lemma \ref{lem:52}(2), relying on 
assertions (2) and (3) in dimension $n - 1$.

\medskip \noindent
(3) Same as proof of Lemma \ref{lem:52}(3). We rely on assertion (2) in 
dimension $n$.
\end{proof}

\begin{lem} \label{lem:65}
Let $K$ be an $n$-dimensional TLF over $\k$, with $n \geq 1$, and let 
$\bsym{\si}$ be a system of liftings for $K$. Let $M_1$ and $M_2$ 
be finite $K$-modules. For any $i$ there is equality
\[ \mrm{E}^{K}_{\bsym{\si}}(M_1, M_2) = 
\mrm{E}^{K}_{\bsym{\si}}(M_1, M_2)_{i, 1} +
\mrm{E}^{K}_{\bsym{\si}}(M_1, M_2)_{i, 2} . \]
\end{lem}

\begin{proof}
For $i = 1$ this is clear. (It is Tate's original observation in \cite{Ta}.)

Assume $i \geq 2$ (and hence also $n \geq 2$). For this we use induction on 
$n$. Choose $K$-linear isomorphisms 
$K^{r_l} \cong M_l$ for $l = 1, 2$. According to Lemmas \ref{lem:52} and 
\ref{lem:66} there are $\k$-linear isomorphisms 
\[ \mrm{E}^{K}_{\bsym{\si}}(M_1, M_2) \cong 
\opn{Mat}_{r_2 \times r_1}(\mrm{E}_{\bsym{\si}}(K)) \]
and 
\[ \mrm{E}^{K}_{\bsym{\si}}(M_1, M_2)_{i, j} \cong 
\opn{Mat}_{r_2 \times r_1}(\mrm{E}_{\bsym{\si}}(K)_{i, j}) . \]
Therefore we can assume that $M_1 = M_2 = K$. 

The induction hypothesis says that the identity automorphism 
$\bsym{1}_{\kk_1(K)}$ of the TLF $\kk_1(K)$ is a sum 
$\bsym{1}_{\kk_1(K)} = \bar{\phi}_1 + \bar{\phi}_2$, where 
$\bar{\phi}_j \in \mrm{E}_{\d_1(\bsym{\si})}(\kk_1(K))_{i - 1, j}$.
Choose a uniformizer $a \in \OO_1(K)$. Any element of $K$ has a unique 
expansion as a series 
$\sum\nolimits_{q \in \Z} \si_1(b_q) \cd a^q$, where 
$b_q \in \kk_1(K)$ and $b_q = 0$ for $q \ll 0$. Define 
$\phi_j \in \opn{End}_{\k}(K)$ by the formula 
\[ \phi_j \Bigl( \sum\nolimits_{q \in \Z} \si_1(b_q) \cd a^q \Bigr) :=
\sum\nolimits_{q \in \Z} \si_1 \bigl( \bar{\phi}_j(b_q) \bigr) \cd a^q . \]
A little calculation shows that $\phi_j \in \mrm{E}_{\bsym{\si}}(K)_{i, j}$; 
and clearly $\phi_1 + \phi_2 = \bsym{1}_K$. 
\end{proof}

Here is a definition from \cite{Ta}. 

\begin{dfn}
Let $M$ be a $\k$-module. An operator
$\phi \in \opn{End}_{\k}(M)$ is called {\em finite potent} if for 
some positive integer $q$, the operator $\phi^q$ has finite rank, i.e.\ the 
$\k$-module $\phi^q(M)$ is finite. 
\end{dfn}

\begin{lem} \label{lem:100}
Let $K$ be a TLF over $\k$ of dimension $n \geq 1$, let 
$\bsym{\si}$ be a system of liftings for $K$, and let $M$ be a finite 
$K$-module. Then any operator 
\[ \phi \in \bigcap_{\substack{i = 1, \ldots, n \\ j = 1, 2}} \ 
\opn{E}^{K}_{\bsym{\si}}(M, M)_{i, j} \]
is finite potent.
\end{lem}

\begin{proof}
The proof is by induction on $n$. (For $n = 1$ this is Tate's original 
observation.) 

Since $\phi \in \opn{E}^{K}_{\bsym{\si}}(M, M)_{1, 1}$, there is a lattice 
$L_2 \in \opn{Lat}(M)$ such that $\phi(M) \subset L_2$. 
Since $\phi \in \opn{E}^{K}_{\bsym{\si}}(M, M)_{1, 2}$, there is a lattice 
$L_1 \in \opn{Lat}(M)$ such that $\phi(L_1) = 0$. 
After replacing $L_1$ by a smaller lattice, we can assume that 
$L_1 \subset L_2$. Consider the commutative diagram 
\[ \UseTips \xymatrix @C=8ex @R=6ex {
0 
\ar[r]^{\subset}
\ar[d]_{\phi}
&
L_1
\ar[d]_{\phi}
\ar[r]^{\subset}
\ar[dl]_{\phi}
&
L_2
\ar[d]_{\phi}
\ar[r]^{\subset}
&
M
\ar[dl]_{\phi}
\ar[d]_{\phi}
\\
0 
\ar[r]_{\subset}
&
L_1
\ar[r]_{\subset}
&
L_2
\ar[r]_{\subset}
&
M
} \]
in $\cat{Mod} \k$. 
Define 
$\bar{M} := L_2 / L_1$. 
If we can prove that the induced homomorphism 
$\bar{\phi} : \bar{M} \to \bar{M}$ is finite potent, then it will follow, by a 
simple linear algebra argument based on the diagram above, that $\phi$ is 
finite potent. 

If $n = 1$ then $\bar{M}$ is finite over $\k$, so we are done. If $n \geq 2$, 
then by definition
\[ \bar{\phi} \in 
\bigcap_{\substack{i = 1, \ldots, n - 1 \\ j = 1, 2}} \ 
\opn{E}^{K}_{\d_1(\bsym{\si})}(\bar{M}, \bar{M})_{i, j} . \]
The induction hypothesis says that $\bar{\phi}$ is finite potent. 
\end{proof}

\begin{thm} \label{thm:65}
Let $K$ be an $n$-dimensional TLF over $\k$, and let 
$(M_1, M_2)$ be a pair of finite $K$-modules.
Suppose $\bsym{\si}$ and $\bsym{\si}'$
are two systems of liftings for $K$. 
\begin{enumerate}
\item There is equality 
\[ \mrm{E}^{K}_{\bsym{\si}}(M_{1}, M_{2}) =
\mrm{E}^{K}_{\bsym{\si}'}(M_{1}, M_{2}) \]
inside $\opn{Hom}_{\k}(M_{1}, M_{2})$
\item If $n \geq 1$, there is equality 
\[ 
\mrm{E}^{K}_{\bsym{\si}}(M_{1}, M_{2})_{i, j} =
\mrm{E}^{K}_{\bsym{\si}'}(M_{1}, M_{2})_{i, j} \]
for all $i = 1, \ldots, n$ and $j = 1, 2$. 
\end{enumerate}
\end{thm}

\begin{proof}
(1) By symmetry it is enough to prove the inclusion ``$\subset$''. 
The proof is by induction of $n$. For $n = 0$ there is nothing to prove.

Now assume $n \geq 1$. Let 
$\phi \in \mrm{E}^{K}_{\bsym{\si}}(M_{1}, M_{2})$. We have to prove that 
$\phi \in \mrm{E}^{K}_{\bsym{\si}'}(M_{1}, M_{2})$.
Since condition (2.i) of Definition \ref{dfn:5} does not involve the liftings, 
there is nothing to check. 

Next we consider condition (2.ii). Take some 
$\phi$-refinement 
$(L'_1, L'_2) \prec_{\phi} (L_1, L_2)$ in $\opn{Lat}(M_1, M_2)$.
Define 
$\bar{M}_1 := L_1 / L'_1$ and $\bar{M}_2 := L'_2 / L_2$,
and let $\bar{\phi} : \bar{M}_1 \to \bar{M}_2$
be the induced homomorphism. 
Let us write  
$\bar{K} := \kk_1(K)$, $\bar{\bsym{\si}} := \d_1(\bsym{\si})$ and 
$\bar{\bsym{\si}}' := \d_1(\bsym{\si}')$.
We know that 
\begin{equation} \label{eqn:65}
\bar{\phi} \in \mrm{E}^{\bar{K}}_{\bar{\bsym{\si}}} \bigl(
\opn{rest}_{\si_1}(\bar{M}_1), \opn{rest}_{\si_1}(\bar{M}_2) \bigr) .
\end{equation}
The induction hypothesis says that 
$\mrm{E}_{\bar{\bsym{\si}}}(\bar{K}) = \mrm{E}_{\bar{\bsym{\si}}'}(\bar{K})$.

Choose $\bar{K}$-linear isomorphisms
$\chi_l : \bar{K}^{r_l} \iso \opn{rest}_{\si_1}(\bar{M}_l)$
and
$\chi'_l : \bar{K}^{r_l} \iso \lb \opn{rest}_{\si'_1}(\bar{M}_l)$.
This gives rise to a commutative diagram 
\[ \UseTips \xymatrix @C=6ex @R=6ex {
\bar{M}_1
\ar[r]^{=}
&
\bar{M}_1
\ar[r]^{\bar{\phi}}
&
\bar{M}_2
\ar[r]^{=}
&
\bar{M}_2
\\
\bar{K}^{r_1}
\ar[u]_{\chi'_1}
\ar[r]^{\psi_1}
&
\bar{K}^{r_1}
\ar[u]_{\chi_1}
\ar[r]^{\psi}
&
\bar{K}^{r_2}
\ar[u]_{\chi_2}
\ar[r]^{\psi_2}
&
\bar{K}^{r_2}
\ar[u]_{\chi'_2}
} \]
in $\cat{Mod} \k$.
According to formula (\ref{eqn:65}) and Lemma \ref{lem:52}, the operator $\psi$ 
is in \lb
$\opn{Mat}_{r_2 \times r_1}(\mrm{E}_{\bar{\bsym{\si}}}(\bar{K}))$.
Combining Lemma \ref{lem:25} and Theorem \ref{thm:20} we see that the operators 
$\psi_l$ belong to $\opn{GL}_{r_l}(\DD^{\mrm{cont}}_{\bar{K} / \k})$. 
Therefore, by Lemma \ref{lem:67}, we get 
$\psi_l \in \opn{GL}_{r_l}(\mrm{E}_{\bar{\bsym{\si}}}(\bar{K}))$.
We conclude that 
$\psi' := \psi_2 \circ \psi \circ \psi_1$ 
is in 
\[ \opn{Mat}_{r_2 \times r_1}(\mrm{E}_{\bar{\bsym{\si}}}(\bar{K})) = 
\opn{Mat}_{r_2 \times r_1}(\mrm{E}_{\bar{\bsym{\si}}'}(\bar{K})) . \]
So by Lemma \ref{lem:52} we have 
\[ \bar{\phi} = \chi'_2 \circ \psi' \circ \chi_1'^{\, -1} \in 
\mrm{E}^{\bar{K}}_{\bar{\bsym{\si}}'} \bigl(
\opn{rest}_{\si'_1}(\bar{M}_1), \opn{rest}_{\si'_1}(\bar{M}_2) \bigr) . \]
This is what we had to prove. 

\medskip \noindent
(2) Again we only prove the inclusion ``$\subset$'', and the proof is by 
induction on $n$. For $i = 1$ the conditions do not involve the liftings, so 
there is nothing to check. Now consider $i \geq 2$ (and hence for $n \geq 2$). 
We assume that the theorem is true for dimension $n - 1$. 
Take some 
$\phi \in \mrm{E}^{K}_{\bsym{\si}}(M_{1}, M_{2})_{i, j}$,
and let 
$(L'_1, L'_2) \prec_{\phi} (L_1, L_2)$ be a $\phi$-refinement in 
$\opn{Lat}(M_1, M_2)$. In the notation of the proof of part (1) 
above, the operator $\psi$ is inside 
$\opn{Mat}_{r_2 \times r_1}(\mrm{E}_{\bar{\bsym{\si}}}(\bar{K})_{i, j})$.
This is because 
\[ \bar{\phi} \in \mrm{E}^{\bar{K}}_{\bar{\bsym{\si}}} \bigl(
\opn{rest}_{\si_1}(\bar{M}_1), \opn{rest}_{\si_1}(\bar{M}_2) \bigr)_{i, j} , \]
and $\mrm{E}_{\bar{\bsym{\si}}}(\bar{K})_{i, j}$ is a 
$2$-sided ideal in the ring $\mrm{E}_{\bar{\bsym{\si}}}(\bar{K})$.
The induction hypothesis tells us that 
$\mrm{E}_{\bar{\bsym{\si}}}(\bar{K})_{i, j} = 
\mrm{E}_{\bar{\bsym{\si}}'}(\bar{K})_{i, j}$.
Therefore the same calculations as above yield
\[ \bar{\phi} \in \mrm{E}^{\bar{K}}_{\bar{\bsym{\si}}'} \bigl(
\opn{rest}_{\si'_1}(\bar{M}_1), \opn{rest}_{\si'_1}(\bar{M}_2) \bigr)_{i, j} \]
as required.
\end{proof}

Taking $M_1 = M_2 := K$ in the theorem we obtain:

\begin{cor} \label{cor:100}
Let $K$ be an $n$-dimensional TLF over $\k$, and let $\bsym{\si}$ and 
$\bsym{\si}'$ be two systems of liftings for $K$. Then 
$\mrm{E}_{\bsym{\si}}(K) = \mrm{E}_{\bsym{\si}'}(K)$.
If $n \geq 1$ then 
$\mrm{E}_{\bsym{\si}}(K)_{i, j} =  \mrm{E}_{\bsym{\si}'}(K)_{i, j}$
for all $i, j$.
\end{cor}

The corollary justifies the next definition. 

\begin{dfn} \label{dfn:66}
Let $K$ be an $n$-dimensional TLF over $\k$.
\begin{enumerate}
\item We define 
\[ \mrm{E}(K) := \mrm{E}_{\bsym{\si}}(K) , \]
where $\bsym{\si}$ is any system of liftings for $K$. 
Elements of $\mrm{E}(K)$ are called  {\em local Beilinson-Tate operators} on 
$K$.
\item Assume $n \geq 1$. For $i \in \{ 1, \ldots, n \}$ and 
$j \in \{ 1, 2 \}$ we define 
\[ \mrm{E}(K)_{i, j} := \mrm{E}^{K}_{\bsym{\si}}(K, K)_{i, j} , \]
where $\bsym{\si}$ is any system of liftings for $K$.
\end{enumerate}
\end{dfn}

Of course when $n= 0$ we have $\opn{E}(K) = \opn{End}_{\k}(K)$, which is not 
interesting. The next theorem summarizes what we know about BT operators in 
dimensions $\geq 1$. Recall the notion of an $n$-dimensional cubically 
decomposed ring of operators on a commutative $\k$-ring $A$, from Definition 
\ref{dfn:85}.

\begin{thm} \label{thm:103}
Let $K$ be an $n$-dimensional TLF over $\k$, with $n \geq 1$. 
\begin{enumerate}
\item The ring of BT operators $\opn{E}(K)$, with its collection of ideals 
$\{ \opn{E}(K)_{i, j} \}$, is an $n$-dimensional cubically decomposed ring of 
operators on $K$.

\item There are inclusions of rings 
\[ K \subset \DD^{\mrm{cont}}_{K / \k} \subset \opn{E}^{}_{}(K) \subset 
\opn{End}^{\mrm{cont}}_{\k}(K) \subset \opn{End}_{\k}(K).  \]
\end{enumerate}
\end{thm}

\begin{proof}
Assertion (1) is a combination of Lemmas \ref{lem:66}, \ref{lem:65} and 
\ref{lem:100}. 
Assertion (2) is a combination of Lemmas \ref{lem:52}, \ref{lem:60} and 
\ref{lem:67}. 
\end{proof}

\begin{rem}
It would be good to have a structural understanding of the objects 
$\opn{E}(K)$ and $\opn{E}(K)_{i, j}$ associated to a TLF $K$. 
For instance, does $\opn{E}(K)$ carry a canonical structure of ST ring, or 
perhaps some ``higher semi-topological structure''?
Such a structure could help in proving Conjecture \ref{conj:202}.
\end{rem}

\begin{rem}
Osipov \cite{Os2} introduced the categories $\cat{C}_n$, $n \in \N$, that fiber 
over $\cat{Mod} \k$. These categories are defined inductively, in a way that 
closely resembles Beilinson's definitions in \cite{Be1}. The paper \cite{BGW} 
introduced the categories $\cat{Tate}_n$ of {\em $n$-Tate spaces}, also fibered 
over $\cat{Mod} \k$. Presumably these two concepts coincide. 

Let $K$ be an $n$-dimensional TLF over $\k$. It seems likely that $K$ should 
have a canonical $\cat{C}_n$-structure, or a canonical 
$\cat{Tate}_n$-structure. Moreover, the subrings 
$\opn{End}_{\cat{C}_n}(K)$,
$\opn{End}_{\cat{Tate}_n}(K)$ and $\opn{E}(K)$ of $\opn{End}_{\k}(K)$ should 
coincide.  

If that is the case, then some of our statements above become similar or 
equivalent to some results in \cite{Os2}. For instance, our Lemma \ref{lem:44} 
corresponds to \cite[Proposition 3]{Os2}. 
\end{rem}

%\cleardoublepage
\section{Residues} \label{sec:resid}

In this section we provide background for Conjecture \ref{conj:200}.
The base ring $\k$ is a perfect field, and it has the discrete topology.

Recall the way the DG ring of separated differential forms 
$\Om^{\mrm{sep}}_{A / \k} = \bigoplus_{i \geq 0} \Om^{i, \mrm{sep}}_{A / \k}$ 
was defined in Section \ref{sec:CDOs}
for any commutative ST $\k$-ring $A$.  
The usual module of K\"ahler differentials $\Om^{i}_{A / \k}$ is a ST 
$\k$-module, with topology induced from the surjection 
$\opn{T}^{i + 1}_{\k}(A) \surj \Om^{i}_{A / \k}$.
Then 
$\Om^{i, \mrm{sep}}_{A / \k} := (\Om^{i}_{A / \k})^{\mrm{sep}}$
is the associated separated ST module. In degree $0$ we have 
$\Om^{0, \mrm{sep}}_{A / \k} = A^{\mrm{sep}}$. 
There is a canonical surjection of DG ST $\k$-rings
\begin{equation} \label{eqn:71}
\tau_A : \Om^{}_{A / \k} \surj \Om^{\mrm{sep}}_{A / \k} 
\end{equation}
which is a topological strict epimorphism. Given any homomorphism $f : A \to B$ 
in the category $\cat{STRing}_{\mrm{c}} \k$, 
there is an induced commutative diagram of DG ST $\k$-rings
\[ \UseTips \xymatrix @C=10ex @R=6ex {
\Om^{}_{A / \k}
\ar[r]^{\Om(f)}
\ar@{->>}[d]_{\tau_A}
&
\Om^{}_{B / \k}
\ar@{->>}[d]^{\tau_B}
\\
\Om^{\mrm{sep}}_{A / \k} 
\ar[r]^{ \Om^{\mrm{sep}}(f) }
&
\Om^{\mrm{sep}}_{B / \k} 
} \]

Let $K$ be an $n$-dimensional TLF over $\k$,
with its DG ST ring of separated differential forms
$\Om^{\mrm{sep}}_{K / \k} = \bigoplus_{i = 0}^n \Om^{i, \mrm{sep}}_{K / \k}$.
In degree $0$ we have 
$\Om^{0, \mrm{sep}}_{K / \k} = K$, since $K$ is separated (in fact it is a 
complete ST $\k$-module). In degree $n$ the $K$-module 
$\Om^{n, \mrm{sep}}_{K / \k}$ is free of rank $1$ with the fine $K$-module 
topology. If $\bsym{a} = (a_1, \ldots, a_n)$ is a system of uniformizers for 
$K$, then the element 
\begin{equation} \label{eqn:70}
\opn{dlog}(\bsym{a}) := a_1^{-1} \cd \d(a_1) \cdots a_n^{-1} \cd \d(a_n)
\end{equation}
is a basis of $\Om^{n, \mrm{sep}}_{K / \k}$. 
Cf.\ Theorem \ref{thm:25} and Example \ref{exa:69}.

There is a theory of trace homomorphisms for separated differential forms.
For any morphism $f : K \to L$ in $\cat{TLF}^n / \k$, there is a homomorphism 
\begin{equation} \label{eqn:73}
\opn{Tr}^{\mrm{TLF}}_{L / K} : \Om^{\mrm{sep}}_{L / \k} \to 
\Om^{\mrm{sep}}_{K / \k} . 
\end{equation}
This is a degree $0$ homomorphism of DG ST 
$\Om^{\mrm{sep}}_{K / \k}$-modules. 
It is uniquely characterized by these properties: it is functorial; in 
degree $0$ it coincides 
with the usual trace $\opn{tr}_{L / K} : L \to K$; and 
\[ \opn{Tr}^{\mrm{TLF}}_{L / K} \circ \opn{dlog} = 
\opn{dlog} \circ \opn{n}_{L / K} \]
as functions $L \to \Om^{1, \mrm{sep}}_{K / \k}$, where 
$\opn{n}_{L / K} : L \to K$ is the usual norm. The homomorphism 
$\opn{Tr}^{\mrm{TLF}}_{L / K}$ is nondegenerate in top degree, in the sense 
that the induced homomorphism 
\[ \Om^{n, \mrm{sep}}_{L / \k} \to 
\opn{Hom}_{K}(L, \Om^{n, \mrm{sep}}_{K / \k}) \]
is bijective. See \cite[Section 2.3]{Ye1}.

In \cite[Section 2.4]{Ye1} we introduced the residue functional for TLFs. Its 
properties are summarized in the following theorem. 

\begin{thm}[\cite{Ye1}] \label{thm:70}
Let $K$ be an $n$-dimensional TLF over $\k$. There is a $\k$-linear 
homomorphism 
\[ \opn{Res}^{\mrm{TLF}}_{K / \k} : \Om^{n, \mrm{sep}}_{K / \k} \to  \k \]
with these properties. 
\begin{enumerate}
\item Continuity\tup{:} the homomorphism $\opn{Res}^{\mrm{TLF}}_{K / \k}$ is 
continuous.
\item Uniformization\tup{:} let $\bsym{a} = (a_1, \ldots, a_n)$ be a system of 
uniformizers for $K$, and let $\k' \to \OO(K)$ be the unique $\k$-ring
lifting of the last residue field $\k' := \kk_n(K)$ into the ring of integers
$\OO(K)$ of $K$. Then for any $b \in \k'$ and any 
$i_1, \ldots, i_n \in \Z$ we have 
\[ \opn{Res}^{\mrm{TLF}}_{K / \k} \bigl( b \cdot
a_1^{i_1} \cdots a_n^{i_n} \cd \opn{dlog}(\bsym{a}) \bigr) = 
\begin{cases}
\opn{tr}_{\k' / \k}(b) & \tup{ if } \ i_1 = \cdots = i_n = 0 
\\
0 & \tup{ otherwise } . 
\end{cases}  \]
\item Functoriality\tup{:} let $f : K \to L$ be a morphism in the category 
$\cat{TLF}^n \k$. Then 
\[ \opn{Res}^{\mrm{TLF}}_{L / \k} = 
\opn{Res}^{\mrm{TLF}}_{K / \k} \ \circ \ 
\opn{Tr}^{\mrm{TLF}}_{L / K} . \]
\item Nondegeneracy\tup{:} the residue pairing 
\[ \bra{-,-}_{\mrm{res}} :  K \times \Om^{n, \mrm{sep}}_{K / \k} \to \k 
\ , \quad
\bra{a, \al}_{\mrm{res}} :=  \opn{Res}^{\mrm{TLF}}_{K / \k}(a \cd \al) \]
is a topological perfect pairing. 

\end{enumerate}

Furthermore, the function $\opn{Res}^{\mrm{TLF}}_{K / \k}$ is the uniquely 
determined by properties \tup{(1)} and \tup{(2)}. 
\end{thm}

\begin{rem} \label{rem:74}
Actually the residue homomorphism $\opn{Res}^{\mrm{TLF}}_{- / -}$ exists in a 
much greater generality. Recall from Remark \ref{rem:30} that there is a 
category $\cat{TLF} \k$ whose objects are TLFs of all dimensions, and there 
are morphisms $f : K \to L$ for $\opn{dim}(K) < \opn{dim}(L)$. 
The category $\cat{TLF}^n  \k$ is a full subcategory of $\cat{TLF} \k$.
In \cite[Section 2.4]{Ye1} we construct a residue homomorphism 
\[  \opn{Res}^{\mrm{TLF}}_{L / K} : \Om^{\mrm{sep}}_{L / \k} \to  
\Om^{\mrm{sep}}_{K / \k} \]
for any morphism $K \to L$ in $\cat{TLF} \k$.
This is a DG ST $\Om^{\mrm{sep}}_{K / \k}$-linear homomorphism of degree 
$-m$, where $m := \opn{dim}(L) - \opn{dim}(K)$, and it has 
properties like those in 
Theorem \ref{thm:70}. When $K = \k$ this is the residue homomorphism 
$\opn{Res}^{\mrm{TLF}}_{L / \k}$ above; and when $m = 0$ this is the trace 
homomorphism:
$\opn{Res}^{\mrm{TLF}}_{L / K} = \opn{Tr}^{\mrm{TLF}}_{L / K}$. 

Another remark is a sign change: the uniformization formula above 
differs from that of \cite[Theorem 2.4.3]{Ye1} by a factor of 
$(-1)^{\binom{n}{2}}$. This is disguised as a permutation of the factors 
of the differential form $\opn{dlog}(t_1, \ldots, t_n)$. Cf.\ also 
\cite[Remark 2.4.4]{Ye1}. Our better acquaintance 
recently with DG conventions dictates the current formula.
\end{rem}

Let $K$ be a TLF over $\k$ of dimension $n \geq 1$. 
The homological algebra / Lie algebra construction of \cite{Be1},
as explained in \cite[Section 3.1]{Br2}, takes 
as input the cubically decomposed ring of BT operators
$\mrm{E}(K)$ from Definition \ref{dfn:66}, and produces the {\em 
Beilinson-Tate residue functional}
\begin{equation} \label{eqn:80}
\opn{Res}^{\mrm{BT}}_{K / \k} : \Om^n_{K / \k} \to \k . 
\end{equation}

Not much is known about this residue functional when $n \geq 2$. 
We have already posed Conjecture \tup{\ref{conj:200}}, comparing 
$\opn{Res}^{\mrm{BT}}_{K / \k}$ to $\opn{Res}^{\mrm{TLF}}_{K / \k}$.
Here is another conjecture. 

\begin{conj} \label{conj:202}
Let $K$ be a TLF over $\k$. 
The $\k$-linear functional $\opn{Res}^{\mrm{BT}}_{K / \k}$ is continuous. 
\end{conj}

It is closely related to the first conjecture. Indeed:

\begin{prop} \label{prop:201}
\begin{enumerate}
\item Conjecture \tup{\ref{conj:200}} implies Conjecture \tup{\ref{conj:202}}.

\item Conjectures \tup{\ref{conj:202}} and \tup{\ref{conj:204}} together imply 
Conjecture \tup{\ref{conj:200}}.
\end{enumerate}
\end{prop}

\begin{proof}
(1) We know that 
$\tau_K : \Om^{n}_{K / \k} \to \Om^{n, \mrm{sep}}_{K / \k}$ 
and $\opn{Res}^{\mrm{TLF}}_{K / \k}$ 
are continuous; cf.\ Theorem \ref{thm:70}(1).

\medskip \noindent 
(2) Assume $\opn{Res}^{\mrm{BT}}_{K / \k}$ is continuous.
Then, since $\k$ is separated, the homomorphism $\opn{Res}^{\mrm{BT}}_{K / \k}$ 
factors via $\tau_K$. It remains to compare the continuous functionals 
\[  \opn{Res}^{\mrm{BT}}_{K / \k} \ , \ \opn{Res}^{\mrm{TLF}}_{K / \k} : \
\Om^{n, \mrm{sep}}_{K / \k}  \to \k  . \]

Conjecture \tup{\ref{conj:204}} says that we can use the results of \cite{Br2}. 
Now according to \cite[Theorem 26(3)]{Br2}, the functional 
$\opn{Res}^{\mrm{BT}}_{K / \k}$ satisfies 
the uniformization condition, i.e.\ condition (2) of Theorem \ref{thm:70}.
Since the $\k$-module spanned by the forms 
\[ b \cdot a_1^{i_1} \cdots a_n^{i_n} \cd \opn{dlog}(\bsym{a}) \]
is dense inside $\Om^{n, \mrm{sep}}_{K / \k}$, and both functionals 
$\opn{Res}^{\mrm{BT}}_{K / \k}$ and $\opn{Res}^{\mrm{TLF}}_{K / \k}$
agree on it, these functionals must be equal.
\end{proof}

To end this section here are some remarks and examples related to the TLF 
residue. 

\begin{rem} \label{rem:77}
The uniqueness of the residue functional 
$\opn{Res}^{\mrm{TLF}}_{K / \k}$ has several other expressions, besides 
properties (1)-(2) of Theorem \ref{thm:70}. 
For simplicity let us assume that $\k$ is infinite and $\kk_n(K) = \k$. 

Here is one alternative characterization. 
Let $G$ be the ``Galois group'' of $K / \k$, namely  
$G := \opn{Aut}_{\cat{TLF}^n \k}(K)$. 
The group $G$ acts on $\Om^{n, \mrm{sep}}_{K / \k}$ by continuous $\k$-linear 
isomorphisms, and hence it acts on 
$\opn{Hom}_{\k}^{\mrm{cont}}(\Om^{n, \mrm{sep}}_{K / \k}, \k)$.
It is not hard to show that 
$\opn{Res}^{\mrm{TLF}}_{K / \k}$ is the only 
$G$-invariant element
$\rho \in \opn{Hom}_{\k}^{\mrm{cont}}(\Om^{n, \mrm{sep}}_{K / \k}, \k)$
that also satisfies 
$\rho(\opn{dlog}(\bsym{a})) = 1$,
where $\bsym{a}$ is any system of uniformizers of $K$.

For the second characterization of the residue functional, let us 
assume that $\opn{char}(\k) = 0$. (This also works in $\opn{char}(\k) = p > 0$, 
but in a more complicated way -- see \cite[Digression 2.4.28]{Ye1}.) Define 
$\opn{H}^n_{\mrm{DR}}(K) := \opn{H}^n(\Om^{\mrm{sep}}_{K / \k})$.
This is a rank $1$ $\k$-module generated by the cohomology class of 
$\opn{dlog}(\bsym{a})$. A calculation shows that  
$\opn{Res}^{\mrm{TLF}}_{K / \k}$ is the only $\k$-linear homomorphism 
$\rho :\Om^{n, \mrm{sep}}_{K / \k} \to \k$
that factors through $\opn{H}^n_{\mrm{DR}}(K)$ (i.e.\ it 
vanishes on $(n - 1)$-coboundaries), and also satisfies 
$\rho(\opn{dlog}(\bsym{a})) = 1$.
\end{rem}

\begin{rem} \label{rem:78}
In dimension $1$ the residue functional on local fields (with its topological 
aspects) was understood a long time ago (cf.\ Serre's book \cite{Se}). 

The first attempt to extend the residue functional to local fields of dimension 
$n \geq 2$ was by Parshin and his school \cite{Pa1, Pa2, Be1, Lo, Pa3}. 
In \cite{Pa1} the case of a surface is discussed, without attempt to 
isolate the resulting $2$-dimensional local field from its geometric origin.
In \cite{Pa3} there is a brief mention of a residue functional on a stand-alone 
$n$-dimensional local field, but without any details whatsoever.
Beilinson \cite{Be1}, quoting \cite{Pa1, Pa2}, 
incorrectly states that the residue functional on an 
$n$-dimensional local field $K$ is independent of the parametrization of $K$ 
(which, according to Theorem \ref{thm:25}, means independent of the topology on 
$K$). 

Lomadze \cite{Lo} studied the setup of a stand-alone $n$-dimensional local 
field in great detail. However, since he misunderstood the role of the topology 
in local fields of dimension $n \geq 2$ (see Remark \ref{rem:31}), the residue 
functional he proposed was not well-defined. To be specific, the paper 
\cite{Lo} claimed that for a local field $K \in \cat{LF}^n \k$ there is a 
$\k$-linear homomorphism, let us denote it by $\opn{res} : \Om^{n}_{K / \k} \to 
\k$, which satisfies continuity, 
uniformization (property (2) of Theorem \ref{thm:70}), and invariance under 
automorphisms of $K$ in $\cat{LF}^n \k$. However this is false for $n \geq 2$ 
and $\opn{char}(\k) = 0$, as was shown by a counterexample in \cite{Ye1}.
We reproduce this counterexample, in an expanded form, in Examples \ref{exa:80}
and \ref{exa:81} below.

In characteristic $p > 0$ the residue functional is indeed well defined on 
the category $\cat{LF}^n \k$. But this is due to the fact, discovered in 
\cite{Ye1}, that the forgetful functor $\cat{TLF}^n \k \to \cat{LF}^n \k$
is an equivalence when $\opn{char}(\k) = p > 0$. See Remark \ref{rem:31}.
\end{rem}

\begin{exa} \label{exa:80}
This is an expanded version of \cite[Example 2.1.24]{Ye1}.
It shows that when $\opn{char}(\k) = 0$ and $n \geq 2$, there cannot be a 
$\k$-linear homomorphism 
$\opn{res} : \Om^{n}_{K / \k} \to \k$, for a local field 
$K \in \cat{LF}^n \k$, which satisfies continuity, uniformization, and  
invariance under automorphisms of $K$ in $\cat{LF}^n \k$. 

Let  $A$ be any commutative ST $\k$-ring. In order to distinguish between an
``abstract'' differential form $\al \in \Om^i_{A / \k}$ and the ``separated'' 
differential form $\tau_A(\al) \in \Om^{i, \mrm{sep}}_{A / \k}$,
we shall write $\bar{\al} := \tau_A(\al)$. Also we denote by $\bar{\d}$ the 
differential operator in the DG ring $\Om^{\mrm{sep}}_{A / \k}$. So 
$\tau_A \circ \d = \bar{\d} \circ \tau_A$,
as $\k$-linear homomorphisms 
$\Om^{i}_{A / \k} \to \Om^{i + 1, \mrm{sep}}_{A / \k}$.
Note that when $A$ itself is separated we have 
$\Om^{0, \mrm{sep}}_{A / \k} = \Om^{0}_{A / \k} = A$.

Since the homomorphism 
$\opn{res} : \Om^{n}_{K / \k} \to \k$
is assumed to be continuous, and $\k$ is separated (because it is discrete),
it follows that $\opn{res}$ factors through $\Om^{n, \mrm{sep}}_{K / \k}$,
and $\opn{res}(\al) = \opn{res}(\bar{\al})$ for any 
$\al  \in \Om^{n}_{K / \k}$.

We shall use the setup of Example \ref{exa:30}. So $\opn{char}(\k) = 0$,
$n = 2$, and $K = \k((t_1, t_2)) = \k((t_2))((t_1))$, 
the standard $2$-dim\-ensional TLF with 
last residue field $\k$. We choose a collection $\{ b_i \}_{i \in I}$ in 
$\k((t_2))$ that is a transcendence basis over the subfield $\k(t_2)$. We 
single out one element of the indexing set, say $i_0 \in I$, and define 
$\si(b_{i_0}) := b_{i_0} + t_1$. For $i \neq i_0$ we let 
$\si(b_{i}) := b_{i}$. This determines an automorphism $f$ of $K$ in the 
category $\cat{LF}^2 \k$. (We already observed in Example \ref{exa:30} that 
$f$ is not continuous). Let us write 
$b := b_{i_0}$; so $f(t_1) = t_1$, $f(t_2) = t_2$ and $f(b) = b + t_1$. 

Define the differential forms 
\[ \al := t_1^{-1} \cd \d(b) \cd t_2^{-1} \cd \d(t_2) , \quad 
\be := t_1^{-1} \cd \d(b + t_1) \cd t_2^{-1} \cd \d(t_2) \]
and  
\[ \ga := t_1^{-1} \cd \d(t_1) \cd t_2^{-1} \cd \d(t_2) = 
\opn{dlog}(t_1, t_2)   \]
in $\Om^{2}_{K / \k}$.
Note that $\be = \al + \ga$ and $\be = f(\al)$. 

Consider the continuous $\k$-linear derivation 
$\frac{\pa}{\pa t_2}$ of $\k((t_2))$. It is dual to the differential form 
$\bar{\d}(t_2) \in \Om^{1, \mrm{sep}}_{\k((t_2)) / \k}$.
Hence, letting 
$b' := \frac{\pa}{\pa t_2} (b) \in \k((t_2))$, 
we have 
$\bar{\d}(b) = b' \cd \bar{\d}(t_2)$ in 
$\Om^{1, \mrm{sep}}_{\k((t_2)) / \k}$. 
Since the inclusion $\k((t_2)) \to K$ is continuous, it follows that 
$\bar{\d}(b) = b' \cd \bar{\d}(t_2)$ in $\Om^{1, \mrm{sep}}_{K / \k}$.
But then $\bar{\d}(b) \cd \bar{\d}(t_2) = 0$ 
in $\Om^{2, \mrm{sep}}_{K / \k}$, from which we deduce that 
$\bar{\al} = 0$ in $\Om^{2, \mrm{sep}}_{K / \k}$.
Therefore $\opn{res}(\al) = \opn{res}(\bar{\al}) =0$. 
On the other hand, 
$\bar{\be} = \bar{\al} + \bar{\ga} = \bar{\ga}$.
And hence 
\[ \opn{res}(\be) = \opn{res}(\bar{\be}) = \opn{res}(\bar{\ga}) =
\opn{res}(\ga) =  1  \]
by the uniformization property. 
We see that $\be = f(\al)$, $\opn{res}(\al) = 0$ and 
$\opn{res}(\be) = 1$.
\end{exa}

\begin{exa} \label{exa:81}
Here is another way to view the previous example. Again $\k$ has characteristic 
$0$. Let $K$ be the local field $\k((t_1, t_2))$. We consider various 
topologies on $K$ that make it into a TLF; namely we are looking at the objects 
in the fiber above $K$ of the forgetful functor 
$F : \cat{TLF}^n \k \to \cat{LF}^n \k$. 
Theorem \ref{thm:25} shows that the group $\opn{Aut}_{\cat{LF}^n \k}(K)$ 
acts transitively on the objects in this fiber. 

The first topology on the local field $K$ is the standard topology of 
$\k((t_1, t_2))$, and we denote the resulting TLF by $K_{\mrm{st}}$.
For the second topology we use the automorphism $f$ from Example \ref{exa:80}.
We take the fine $(K_{\mrm{st}}, f^{-1})$-module topology on $K$, and call the 
resulting TLF $K_{\mrm{nt}}$. Thus $f : K_{\mrm{nt}} \to K_{\mrm{st}}$
is an isomorphism in $\cat{TLF}^n \k$, and 
$F(K_{\mrm{nt}}) = F(K_{\mrm{st}}) = K$ in $\cat{LF}^n \k$.

Let $K_{\mrm{t}}$ be any TLF such that $F(K_{\mrm{t}}) = K$
(for instance the standard TLF $K_{\mrm{st}}$ and the nonstandard TLF 
$K_{\mrm{nt}}$). There is a surjection
\[ \tau_{\mrm{t}} = \tau_{K_{\mrm{t}}} : 
\Om^2_{K / \k} = \Om^2_{K_{\mrm{t}} / \k} \surj 
\Om^{2, \mrm{sep}}_{K_{\mrm{t}} / \k} , \]
and thus a residue homomorphism 
$\opn{res}_{\mrm{t}} : \Om^2_{K / \k} \to \k$
defined by 
$\opn{res}_{\mrm{t}} :=  
\opn{Res}^{\mrm{TLF}}_{K_{\mrm{t}} / \k} \circ \, \tau_{\mrm{t}}$.

Consider the differential forms $\al, \be, \ga \in \Om^2_{K / \k}$
from Example \ref{exa:80}.
The calculation there shows that 
$\tau_{\mrm{st}}(\al) = 0$. On the other hand,
since $f \circ \tau_{\mrm{nt}} = \tau_{\mrm{st}}  \circ f$
and $f(\ga) = \ga$, we have
\[ f(\tau_{\mrm{nt}}(\al)) = \tau_{\mrm{st}}(f(\al)) =
\tau_{\mrm{st}}(\be) = \tau_{\mrm{st}}(\al) + \tau_{\mrm{st}}(\ga) =
\tau_{\mrm{st}}(\ga) = f(\tau_{\mrm{nt}}(\ga)) , \]
and therefore 
$\tau_{\mrm{nt}}(\al) = \tau_{\mrm{nt}}(\ga)$. 
We conclude that for the differential form $\al \in \Om^2_{K / \k}$,
we have $\opn{res}_{\mrm{st}}(\al) = 0$, but 
$\opn{res}_{\mrm{nt}}(\al) = \opn{res}_{\mrm{nt}}(\ga) = 1$.
\end{exa}

\begin{que} \label{que:80}
Take any $n \geq 2$. Consider the local field 
$K := \k((t_1, \ldots, t_n))$, and the various TLFs 
$K_{\mrm{t}}$ lying above it in $\cat{TLF}^n \k$,
as in the previous example. We know that the residue 
$\opn{res}_{\mrm{t}}(\al)$, for $\al \in \Om^n_{K / \k}$,
could change as we change the topology. However our counterexample involved 
transcendentals (the element $b$). 

What about the subfield $\k(t_1, \ldots, t_n) \in K$~?
Is it true that for a form 
$\al \in \Om^n_{\k(t_1, \ldots, t_n) / \k}$
the residue $\opn{res}_{\mrm{t}}(\al)$ is independent of the topology
$K_{\mrm{t}}$ on $K$~?
\end{que}

%\cleardoublepage
\section{Geometry: Completions} \label{sec:geom}

In this section we give background for Conjecture \ref{conj:204}.
We recall some facts on the Beilinson completion operation, and reproduce 
Beilinson's geometric definition of the BT operators. 

Throughout this section $\k$ is a noetherian commutative ring, and $X$ is 
a finite type $\k$-scheme. By a chain of points of length $n$ in $X$ we mean a 
sequence $\xi = (x_0, \ldots, x_n)$ of points in $X$, such that $x_i$ is a 
specialization of $x_{i - 1}$ for all $i$. The chain $\xi$ is called a {\em 
saturated chain} if every $x_i$ is an immediate specialization of $x_{i - 1}$, 
namely the closed set $\ol{ \{ x_i \}}$ has codimension $1$ in 
$\ol{ \{ x_{i - 1} \}}$.
If $n \geq 1$, we denote by $\d_0(\xi)$ the chain gotten from 
$\xi$ by deleting the point $x_0$. 

Let $\MM$ be a quasi-coherent $\OO_X$-module. Beilinson \cite{Be1} introduced 
the completion $\MM_{\xi}$ of $\MM$ along $\xi$, which we refer to as the {\em 
Beilinson completion}. This is a very special case of 
his higher adeles. The definition of $\MM_{\xi}$ is inductive on $n$, by an 
$n$-fold zig-zag of inverse and direct limits. For a detailed account 
see \cite[Section 3]{Ye1} or \cite{Mo}. 
A basic geometric fact used in the definition is that 
for any coherent sheaf $\MM$, point $x \in X$, and number $i \in \N$, the 
truncated localization $\MM_{x} / \m_{x}^{i+1} \MM_x$, when viewed as an 
$\OO_X$-module supported on the closed set $\ol{ \{ x \}}$, is quasi-coherent.
An important instance of this is when $\MM = \OO_X$ and $i = 0$, 
which gives the residue field 
$\kk(x) = \OO_{X, x} / \m_x$. 

Here are some important properties of the Beilinson completion operation. 
Let $\MM$ be some quasi-coherent $\OO_X$-module and let 
$\xi =  (x_0, \ldots, x_n)$ be a chain in $X$. 
We can view the completion $\MM_{\xi}$ either as a module over the local ring 
$\OO_{X, x_n}$, or as a constant $\OO_X$-module supported on the closed set 
$\ol{ \{ x_n \}}$. Warning: $\MM_{\xi}$ is usually not quasi-coherent.
For any subchain $\xi' \subset \xi$ there is a canonical 
homomorphism $\MM_{\xi'} \to \MM_{\xi}$.
When $n = 0$, so $\xi = (x_0)$, there is a canonical homomorphism 
$\MM_{x_0} \to \MM_{(x_0)}$, where the former is the stalk at the point.
If $\MM$ is coherent, then the homomorphism 
$\what{\MM}_{x_0} \to \MM_{(x_0)}$ from the $\m_{x_0}$-adic completion 
is an isomorphism. 
 
The completion $\OO_{X, \xi}$ of the structure sheaf $\OO_X$ 
is a commutative ring, the  canonical sheaf homomorphism $\OO_{X} \to 
\OO_{X, \xi}$ is flat, and $\MM_{\xi}$ is an $\OO_{X, \xi}$-module. 
The sheaf homomorphism 
$\OO_{X, \xi} \ot_{\OO_{X}} \MM \to \MM_{\xi}$
is an isomorphism. Thus the functor $\MM \mapsto \MM_{\xi}$ is exact. 
If $\MM$ is coherent, $\xi$ is saturated, and $n \geq 1$, then the canonical 
homomorphism
\[ \OO_{X, x_0} \ot_{\OO_{X}} \MM_{\d_0(\xi)}  \to \MM_{\xi}  \]
is an isomorphism. 

The zig-zag completion operation endows $\MM_{\xi}$ with a $\k$-linear 
topology, similar to the iterated Laurent series construction in Definition 
\ref{dfn:41}. The ring $\OO_{X, \xi}$ becomes a ST $\k$-ring, and $\MM_{\xi}$ 
is 
a ST $\OO_{X, \xi}$-module. 

Let $A$ be a semi-local commutative ring, with Jacobson radical $\r$. We say 
that $A$ is a {\em complete semi-local ring} if the canonical homomorphism 
$A \to \lim_{\leftarrow i} A / \r^i$ is bijective. 
The {\em residue ring} of $A$ is the ring $A / \r$, which is a finite product 
of fields. 

\begin{thm}[\cite{Pa1}, \cite{Be1}, \cite{Ye1}] \label{thm:50}
Let $\k$ be an excellent noetherian ring, let $X$ be a finite type $\k$-scheme, 
and let $\xi = (x_0, \ldots, x_n)$ be a saturated chain in $X$ of length 
$n \geq 1$, such that $x_n$ is a closed point. 
Then the Beilinson completions $\OO_{X, \xi}$ and $\kk(x_0)_{\xi}$
have these algebraic properties\tup{:}
\begin{enumerate}
\item The ring $\kk(x_0)_{\xi}$ is a finite product of $n$-dimensional 
local fields over $\k$.
\item The ring $\OO_{X, \xi}$ is a complete semi-local commutative $\k$-ring, 
with Jacobson radical $\r = \OO_{X, \xi} \ot_{\OO_{X, x_0}} \m_{x_0}$
and residue ring $\kk(x_0)_{\xi}$.
\item Let $K$ be one of the factors of the reduced artinian semi-local ring 
$\kk(x_0)_{\xi}$, which by \tup{(1)} is an $n$-dimensional local field. 
The DVR $\OO_1(K)$ is the integral closure in $K$ of the ring 
$\OO_{X, \d_0(\xi)}$.
\end{enumerate}
If the base ring $\k$ is a perfect field, then the completion $\kk(x_0)_{\xi}$
also has these topological properties\tup{:}
\begin{enumerate} 
\setcounter{enumi}{3}
\item Let $K$ be one of the factors of the ring $\kk(x_0)_{\xi}$.
Then $K$, with the induced topology from $\kk(x_0)_{\xi}$, is an 
$n$-dimensional TLF over $\k$. 
\item The image of the field $\kk(x_0)$ in the ST ring $\kk(x_0)_{\xi}$ is 
dense. 
\end{enumerate}
\end{thm}

\begin{proof}
(1-3) For $n = 1$ this is classical. For $n = 2$ this is in \cite{Pa1}. 
For $n \geq 3$ these assertions appear in \cite{Be1} without a 
proof. The proofs are \cite[Theorem 3.3.2]{Ye1} and 
\cite[Corollary 3.3.5]{Ye1}.

\medskip \noindent
(4-5) For $n = 1$ this is classical. For $n \geq 2$ these assertions 
are  \cite[Proposition 3.3.6]{Ye1} and \cite[Corollary 3.3.7]{Ye1}
\end{proof}

\begin{rem} \label{rem:71}
The condition that $x_n$ is a closed point is only important to ensure 
that the last residue fields $\kk_n(K)$ are finite over $\k$. Cf.\ Remark 
\ref{rem:70}. The results in \cite{Ye1} quoted in the proof above only require
the chain $\xi$ to be saturated.
\end{rem}

Suppose $\xi = (x_0, \ldots, x_n)$ is a saturated chain in $X$. We have 
seen that there is a commutative diagram of flat ring homomorphisms 
\[ \UseTips \xymatrix @C=5ex @R=5ex {
\OO_{X, x_n}
\ar[r]
\ar[d]
&
\OO_{X, (x_n)}
\ar[dr]
\\
\OO_{X, x_0}
\ar[r]
&
\OO_{X, (x_0)}
\ar[r]
&
\OO_{X, \xi}
} \]

\begin{dfn}
Let $\xi = (x_0, \ldots, x_n)$ be a saturated chain in $X$ of length $n \geq 1$,
and let $M$ be a finite length $\OO_{X, x_0}$-module. An $\OO_{X, x_1}$-lattice 
in $M$ is a finite $\OO_{X, x_1}$-submodule $L$ of $M$, such that 
$M = \OO_{X, x_0} \cd L$. We denote by 
$\opn{Lat}_{X, \xi}(M)$ the set of all such lattices. 
\end{dfn}

Of course the points $x_2, \ldots, x_n$ have no influence on 
$\opn{Lat}_{X, \xi}(M)$.
Note that if $\xi$ has length $0$ then $M_{\xi} = M$ for any finite length 
$\OO_{X, x_0}$-module $M$. 

\begin{lem}
Let $\xi = (x_0, \ldots, x_n)$ be a saturated chain in $X$, and 
let $M$ be a finite length $\OO_{X, x_0}$-module. 
If $L, L' \in \opn{Lat}_{X, \xi}(M)$ and $L \subset L'$,
then $L' / L$ is a finite length $\OO_{X, x_1}$-module.
\end{lem}

\begin{proof}
We can assume that $M \neq 0$. 
Let $Z$ be the support in $\opn{Spec} \OO_{X, x_1}$ of $L$. Then $Z$ is a 
$1$-dimensional scheme, with only two points: the closed point $x_1$ and the 
generic point $x_0$. The finite $\OO_{X, x_1}$-module
$L' / L$ satisfies 
\[ (L' / L)_{x_0} \cong \OO_{X, x_0} \ot_{\OO_{X, x_1}} (L' / L) = 0 , \]
and hence it is supported on $\{ x_1 \}$. 
\end{proof}

Let $\xi = (x_0, \ldots, x_n)$ be a saturated chain in $X$, and let 
$M$ be a finite length $\OO_{X, x_0}$-module. We can 
view $M$ as a quasi-coherent sheaf on $X$, constant on the closed set 
$\ol{\{ x_0 \}}$. The canonical homomorphism 
$M_{\d_0 (\xi)} \to M_{\xi}$ is bijective. (If $n = 0$ then $\d_0 (\xi)$ is 
empty, and we define $M_{()} := M$). 
Note that $M_{\xi}$ is a finite length $\OO_{X, \xi}$-module.

Suppose we are given $\OO_{X, x_1}$-lattices $L \subset L'$ in $M$.
By the exactness of completion there are inclusions 
\[ L_{\d_0 (\xi)} \subset L'_{\d_0 (\xi)} \subset M_{\d_0 (\xi)} = 
M_{\xi} , \]
and there is  a canonical isomorphism of finite length 
$\OO_{X, \d_0 (\xi)}$-modules
\[ (L' / L)_{\d_0 (\xi)} \cong  L'_{\d_0 (\xi)}  / L_{\d_0 (\xi)} . \]

Let $(M_1, M_2)$ be a pair of finite length $\OO_{X, x_0}$-modules. 
We denote by \lb $\opn{Lat}_{X, \xi}(M_1, M_2)$ the set of pairs 
$(L_1, L_2)$, where $L_i \in \opn{Lat}_{X, \xi}(M_i)$.
We write $M_{i, \xi} := (M_i)_{\xi}$.
Suppose $\phi : M_{1, \xi} \to M_{2, \xi}$ is a $\k$-linear operator. 
Like in Definition \ref{dfn:50}, we say that 
$(L'_1, L'_2)$ is a $\phi$-refinement of $(L_1, L_2)$, and that 
$(L'_1, L'_2) \prec_{\phi} (L_1, L_2)$ is a $\phi$-refinement in 
$\opn{Lat}_{X, \xi}(M_1, M_2)$, if 
$L'_1 \subset L_1$, $L_2 \subset L'_2$, 
$\phi(L_{1, \d_0 (\xi)}) \subset L'_{2, \d_0 (\xi)}$ and 
$\phi(L'_{1, \d_0 (\xi)}) \subset L_{2, \d_0 (\xi)}$.

Suppose $A$ is a semi-local ring, with residue ring $K$. Any finite length 
$A$-module $M$ has a canonical decomposition $M = \bigoplus_{\n} M_{\n}$, where 
$\n$ runs over the finite set of maximal ideals of 
$A$, which of course coincides with the set $\opn{Spec} K$.

\begin{dfn}
Let $A$ be a semi-local ring in $\cat{Ring}_{\mrm{c}} \k$,
with residue ring $K$. Let $M_1, M_2$ be finite length $A$-modules, and let 
$\phi : M_1 \to M_2$ be a $\k$-linear homomorphism. We say that $\phi$ is 
{\em local on $\opn{Spec} K$} if 
$\phi(M_{1, \n}) \subset M_{2, \n}$ for every $\n \in \opn{Spec} K$.
\end{dfn}

Here is a slight enhancement of the original definition found in \cite{Be1}; 
cf.\ Remark \ref{rem:50} below and Definition \ref{dfn:5}.

\begin{dfn}[\cite{Be1}] \label{dfn:1}
Let $\xi = (x_0, \ldots, x_n)$ be a saturated chain of points in $X$, such that 
$x_n$ is a closed point. Let $(M_1, M_2)$ be a pair of finite length modules 
over the ring $\OO_{X, x_0}$. We define the subset
\[ \opn{E}_{X, \xi}(M_1, M_2) \subset 
\opn{Hom}_{\k}^{}(M_{1, \xi}, M_{2, \xi})  \]
as follows. 

\begin{enumerate}
\item If $n = 0$, then any $\k$-linear homomorphism 
$\phi :  M_{1, \xi} \to M_{2, \xi}$ belongs to \lb
$\opn{E}_{X, \xi}(M_1, M_2)$.
\item If $n \geq 1$, then a $\k$-linear homomorphism 
$\phi :  M_{1, \xi} \to M_{2, \xi}$ belongs to \lb
$\opn{E}_{X, \xi}(M_1, M_2)$ if it satisfies these three conditions:
\begin{enumerate}
\rmitem{i} Every $(L_1, L_2) \in \opn{Lat}_{X, \xi}(M_1, M_2)$
has some $\phi$-refinement $(L'_1, L'_2)$.
\rmitem{ii} For every $\phi$-refinement 
$(L'_1, L'_2) \prec_{\phi} (L_1, L_2)$ in $\opn{Lat}_{X, \xi}(M_1, M_2)$
the induced homomorphism 
\[ \bar{\phi} : (L_1 / L'_1)_{\d_0(\xi)} \to (L'_2 / L_2)_{\d_0(\xi)}  \]
belongs to 
\[ \mrm{E}_{X, \d_0(\xi)} ( L_1 / L'_1, L'_2 / L_2)  . \] 
\rmitem{iii} The homomorphism $\phi$ is local on $\opn{Spec} \kk(x_0)_{\xi}$.
\end{enumerate}
\end{enumerate}
\end{dfn}

\begin{rem} \label{rem:50}
Condition (2.iii) of Definition \ref{dfn:1} is not part of the original 
definition in \cite{Be1}. 
Note that Tate  \cite{Ta} only considered smooth curves, for which the 
completion is always a single local field, and there is no issue of locality.

The same locality condition eventually appears in Braunling's treatment -- 
see the definition of the ring $E_j$ in \cite[Theorem 26(1)]{Br2}.
\end{rem}

The next definition uses notation like that of Tate. It can of course be 
rewritten using the notations of \cite{Be1} or of \cite{Br1, Br2}. 
Compare to Definition \ref{dfn:73} above.

\begin{dfn}[\cite{Be1}] \label{dfn:3}
Let $\xi = (x_0, \ldots, x_n)$ be a saturated chain of points in $X$ of length 
$n \geq 1$, such that 
$x_n$ is a closed point. Let $(M_1, M_2)$ be a pair of finite length modules 
over the ring $\OO_{X, x_0}$. For any $i \in \{ 1, \dots, n \}$ and 
$j \in \{ 1, 2 \}$ we define the subset
\[ \opn{E}_{X, \xi}(M_1, M_2)_{i, j} \subset \opn{E}_{X, \xi}(M_1, M_2) \]
to be the set of operators $\phi : M_{1, \xi} \to M_{2, \xi}$ in 
$\opn{E}_{X, \xi}(M_1, M_2)$ 
that satisfy the conditions below. 
\begin{itemize}
\rmitem{i} The operator $\phi$ belongs to 
$\opn{E}_{X, \xi}(M_1, M_2)_{1, 1}$ if there 
exists some $L_2 \in \lb \opn{Lat}_{X, \xi}(M_2)$ such that 
$\phi(M_{1, \xi}) \subset L_{2, \d_0(\xi)}$. 
\rmitem{ii}  The operator $\phi$ belongs to 
$\opn{E}_{X, \xi}(M_1, M_2)_{1, 2}$
if there exists some $L_1 \in \lb \opn{Lat}_{X, \xi}(M_1)$ such that 
$\phi(L_{1, \d_0(\xi)}) = 0$. 
\rmitem{iii} This only refers to $n \geq 2$. 
For  $i \in \{ 2, \dots, n \}$ and 
$j \in \{ 1, 2 \}$, the operator $\phi$ belongs to 
$\opn{E}_{X, \xi}(M_1, M_2)_{i, j}$
if for any $\phi$-refinement
$(L'_1, L'_2) \prec_{\phi} (L_1, L_2)$ in $\opn{Lat}_{X, \xi}(M_1, M_2)$,
the induced homomorphism 
\[ \bar{\phi} : (L_1 / L'_1)_{\d_0(\xi)} \to (L'_2 / L_2)_{\d_0(\xi)}  \]
belongs to 
\[  \opn{E}_{X, \d_0(\xi)} (L_1 / L'_1, L'_2 / L_2)_{i - 1, j}  . \] 
\end{itemize}
\end{dfn}

\begin{dfn} \label{dfn:104}
Let $\xi = (x_0, \ldots, x_n)$ be a saturated chain of points in $X$,
of length $n \geq 1$, such that $x_n$ is a closed point.
Consider the residue field $K := \kk(x_0)$. 
\begin{enumerate}
\item  We define  
$\opn{E}_{X, \xi}(K) := \opn{E}_{X, \xi}(K, K)$.
\item If $n \geq 1$ we define 
$\opn{E}_{X, \xi}(K)_{i, j} := \opn{E}_{X, \xi}(K, K)_{i, j}$.
\end{enumerate}
\end{dfn}

By definition there are inclusions
\[ \opn{E}_{X, \xi}(K)_{i, j} \subset \opn{E}_{X, \xi}(K) \subset 
\opn{End}_{\k}(K_{\xi}) . \]

\begin{thm}[{\cite{Be1}, \cite[Proposition 13]{Br2}}] \label{thm:104}
Assume $\k$ is a perfect field. The data 
\[ \bigl( \mrm{E}_{X, \xi}(K), \{ \mrm{E}_{X, \xi}(K)_{i, j} \} \bigr) \]
from Definition \tup{\ref{dfn:104}}
is an $n$-dimensional cubically decomposed ring of operators on $K_{\xi}$,
in the sense of Definition \tup{\ref{dfn:85}}.
\end{thm}

Conjecture \ref{conj:204} asserts that this $n$-dimensional cubically 
decomposed 
ring of operators on $K_{\xi}$ coincides with the cubically decomposed 
ring of operators
\[  \bigl( \mrm{E}(K_{\xi}), \{ \mrm{E}(K_{\xi})_{i, j} \} \bigr) \]
from Definition \ref{dfn:66}, modified as in formula (\ref{eqn:100}).

\begin{rem} \label{rem:200}
Consider an integral finite type $\k$-scheme $X$ of dimension $n$. 
Let $\xi = (x_0, \ldots, x_n)$ be a maximal chain in $X$; so $x_0$ is the 
generic point. Write $K := \kk(X) = \kk(x_0)$. According to Theorem 
\ref{thm:104} there is a cubically decomposed ring of operators 
$\opn{E}_{X, \xi}(K)$ on $K_{\xi}$. Applying the abstract BT residue of 
formula (\ref{eqn:88}) with $E := \opn{E}_{X, \xi}(K)$, we obtain the functional
\[ \opn{Res}^{\mrm{BT}}_{X, \xi} :=
\opn{Res}^{\mrm{BT}}_{K_{\xi} / \k; E} : \Om^{n}_{K / \k} \to \k . \]
This is the residue functional that Beilinson had in \cite{Be1}. 

Beilinson claimed in \cite{Be1} that the functionals 
$\opn{Res}^{\mrm{BT}}_{X, \xi}$ satisfy several geometric properties. Most 
notably, when $X$ is a proper scheme, then for any $\al \in \Om^{n}_{K / \k}$
there is a global residue formula:
\begin{equation} \label{eqn:900}
\sum_{\xi} \opn{Res}^{\mrm{BT}}_{X, \xi}(\al) = 0 .
\end{equation}
The sum is on all maximal chains $\xi$ in $X$. 

Conjectures \ref{conj:200} and \ref{conj:204}, combined with our results in 
\cite{Ye1} regarding the residue functionals 
$\opn{Res}^{\mrm{TLF}}_{K_{\xi} / \k}$, imply most of the geometric properties 
of the residue functionals $\opn{Res}^{\mrm{BT}}_{X, \xi}$ stated in \cite{Be1}, 
including formula (\ref{eqn:900}). 

Conversely, as noted by Beilinson (private communication), a direct proof of 
the geometric properties of the functionals 
$\opn{Res}^{\mrm{BT}}_{X, \xi}$ (perhaps by generalizing Tate's 
original idea to higher dimensions), together with Conjecture 
\ref{conj:204}, would imply Conjecture \ref{conj:200}.
\end{rem}

%\cleardoublepage

\end{document}